\documentclass[a4paper,11pt]{amsart}

\usepackage{amsfonts, latexsym, amsmath, amssymb, amsthm, amscd}
\usepackage[all]{xy}
\usepackage[english]{babel}
\usepackage[utf8]{inputenc}
\usepackage{graphicx}
\usepackage{booktabs}
\usepackage{array, tabularx}
\usepackage{textcomp}
\usepackage{tikz, tikz-cd}
\usepackage{multirow}
\usepackage{paralist}
\usepackage{comment}
\usepackage{url}
\usepackage{tensor}
\usepackage{MnSymbol}
\usepackage{color}
\usepackage[hypertexnames=false,
backref=page,
    pdftex,
    pdfpagemode=UseNone,
    breaklinks=true,
    extension=pdf,
    colorlinks=true,
    linkcolor=blue,
    citecolor=blue,
    urlcolor=blue,
]{hyperref}

\usepackage{enumitem}
\usepackage[top=1.5in, bottom=.9in, left=0.9in, right=0.9in]{geometry}

\theoremstyle{plain}
\newtheorem{theorem}{Theorem}[section]
\newtheorem*{theorem*}{Theorem}
\newtheorem{definition}[theorem]{Definition}
\newtheorem{lemma}[theorem]{Lemma}
\newtheorem{prop}[theorem]{Proposition}
\newtheorem*{prop*}{Proposition}

\newtheorem{rem}[theorem]{Remark}

\newtheorem*{mt*}{Main Theorem}

\def\Real{{\mathfrak{R}}{\mathfrak{e}}\,}

\DeclareMathOperator{\Span}{Span}
\DeclareMathOperator{\id}{id}

\DeclareMathOperator{\Ker}{Ker}

\DeclareMathOperator{\im}{Im}
\DeclareMathSymbol{\Finv} {\mathord}{AMSb}{"60}

\newcommand\restrict[1]{\raisebox{-.5ex}{$|$}_{#1}}

\newcommand{\R}{\mathbb{R}}
\newcommand{\C}{\mathbb{C}}

\newcommand{\q}{\mathbb{Q}}
\newcommand{\del}{\partial}
\newcommand{\delbar}{\overline{\partial}}

\numberwithin{equation}{section}

\let\phi\varphi

\DeclareFontFamily{U}{MnSymbolC}{}
\DeclareSymbolFont{MnSyC}{U}{MnSymbolC}{m}{n}
\DeclareFontShape{U}{MnSymbolC}{m}{n}{
    <-6>  MnSymbolC5
   <6-7>  MnSymbolC6
   <7-8>  MnSymbolC7
   <8-9>  MnSymbolC8
   <9-10> MnSymbolC9
  <10-12> MnSymbolC10
  <12->   MnSymbolC12}{}
\DeclareMathSymbol{\intprod}{\mathbin}{MnSyC}{'270}

\allowdisplaybreaks

\author{Tommaso Sferruzza}
\address[Tommaso Sferruzza]{
Dipartimento di Scienze Matematiche, Fisiche e Informatiche\\
Unità di Mate\-matica e Informatica\\
Università degli studi di Parma}
\email{tommaso.sferruzza@unipr.it}

\author{Adriano Tomassini}
\address[Adriano Tomassini]{
Dipartimento di Scienze Matematiche, Fisiche e Informatiche\\
Unità di Mate\-matica e Informatica\\
Università degli studi di Parma}
\email{adriano.tomassini@unpir.it}

\title[On cohomological and formal properties of strong K\"ahler with torsion ...]{On cohomological and formal properties of strong K\"ahler with torsion and astheno-K\"ahler metrics}

\thanks{The first author has been supported by GNSAGA of INdAM. The second author has been supported by the project PRIN2017 “Real and Complex Manifolds: Topology, Geometry and holomorphic dynamics” (code 2017JZ2SW5), and by GNSAGA of INdAM}

\date{\today}

\begin{document}

\begin{abstract}
We provide families of compact {\em astheno-K\"ahler} nilmanifolds and we study the behaviour of the complex blowup of such manifolds. We prove that the existence of an astheno-K\"ahler metric satisfying an extra differential condition is not preserved by blowup. We also study the interplay between {\em Strong K\"ahler with torsion metrics} and geometrically Bott-Chern metrics. We show that Fino-Parton-Salamon nilmanifolds are geometrically-Bott-Chern-formal, whereas  we obtain negative results on the product of two copies of {\em primary Kodaira surface} , {\em Inoue surface of type $\mathcal{S}_M$} and on the product of a Kodaira surface with an Inoue surface.
\end{abstract}

\maketitle
\tableofcontents
\section{Introduction}
Compact K\"ahler manifolds share many remarkable cohomological and metric properties, e.g., they satisfy the $\del\delbar$-Lemma, all the complex cohomologies, Dolbeault, Aeppli and Bott-Chern, are isomorphic, their underlying structure of smooth manifold is formal in the sense of Sullivan, they satisfy the Hard Lefschetz Condition. As a consequence, the existence of a K\"ahler structure on a compact smooth manifold $M$ forces many topological restraints, e.g., the Betti numbers of even index are positive, the Betti numbers of odd index are even and  Massey products vanish. On the other hand, a large and natural class of compact complex manifolds arises as a quotient of  simply connected nilpotent Lie groups by a discrete uniform subgroup; nevertheless, in  view of Benson and Gordon \cite{BG}, the underlying structure of the smooth manifold of such compact quotients admits a K\"ahler structure if and only if the group is Abelian, that is the manifold is diffeomorphic to a torus, and, more in general, by Hasegawa (see \cite{Ha1} and also \cite{Ha2}) a compact quotient of a simply-connected solvable Lie group by a closed subgroup carries a K\"ahler structure if and only if it is a finite quotient of a complex torus. Furthermore, such  a class of compact manifolds admits special Hermitian metrics, e.g., 
{\em Strong K\"ahler with torsion} metrics or {\em pluriclosed} (see \cite{Bi} and \cite{GHR}), {\em astheno-K\"ahler} metrics in the sense of Jost and Yau (see \cite{JY}). More precisely, a Hermitian metric $g$ with fundamental form $F$ on an $n$-dimensional complex manifold $(M,J)$ is said to be
\begin{itemize}
\item  {\em strong K\"ahler with torsion}, or shortly {\em SKT}, if $\del\delbar F=0$ or, equivalently $dd^cF=0$;\\
\item  {\em astheno-K\"ahler}, if $\del\delbar F^{n-2}=0$ or, equivalently $dd^c F^{n-2}=0$,
\end{itemize}
where $d^c=J^{-1}dJ$.
The SKT metrics have been studied by many
authors and they have also applications in type II string theory
and in $2$-dimensional supersymmetric $\sigma$-models
\cite{GHR,Str,IP}. Moreover, they have also relations with generalized
K\"ahler geometry (see for instance
\cite{GHR,Gu,Hi2,AG,CG}). \newline 

Some rigidity theorems concerning compact
astheno-K\"ahler manifolds have been showed in \cite[Theorem
6]{JY} and in \cite{LYZ}, where, in particular, a generalization
to higher dimension of the Bogomolov's Theorem on $VII_0$ surfaces
is proved (see \cite[Corollary 3]{LYZ}). Astheno-K\"ahler structures on Calabi-Eckmann manifolds have been constructed in \cite{Matsuo}. For other results on SKT and astheno-K\"ahler metrics and special metrics on complex manifolds see also \cite{FPS,Ca, P,IP,MT,FG,FT2,FT,Sw,RT,OOS,PS,FGV} and the references therein.

From the cohomological point of view, beside to the Dolbeault cohomology groups, a useful tool in the study of compact complex (non-K\"ahler) manifolds is provided by {\em Bott-Chern} and {\em Aeppli} cohomology groups. Indeed, the property of satisfying the {\em $\del\delbar$}-Lemma can be characterized in terms of their dimensions (see \cite{AT12}). By adapting the construction of Massey triple products respectively Dolbeault Massey triple products as in \cite{Sul1,DGMS}, respectively in Neisendorfer and Taylor (see \cite{NT}), in \cite{AT15} the notion of {\em triple Aeppli-Bott-Chern-Massey} products are introduced with the aim to give obstructions to the existence of  geometrically-Bott-Chern-formal metrics, that is Hermitian metrics whose space of Bott-Chern harmonic forms has the structure of algebra. The notion of {\em formal Riemannian metrics} has been previously introduced and studied intensively by Kotschick in \cite{Kot}.\newline
Very recently Milivojevi\'{c} and Stelzig in \cite{MS} introduced the {\em n-fold Massey products} in the spectral sequence sense, for a commutative bigraded differential algebra. In particular, for $n=3$, they recover the definition of extending the definition triple Aeppli-Bott-Chern-Massey products as in \cite{AT15}. Furthermore, they consider the notions of  {\em weak formality}, respectively {\em strong formality} which have applications in the study of cohomological properties of complex non K\"ahler manifolds. For other results on these topics we refer to \cite{St, ST, St1, TTo}
\vskip.2truecm
In the first part of the present paper we are interested in providing astheno-K\"ahler metrics on compact nilmanifolds, endowed with left-invariant complex structures and in studying the behaviour of blowup of compact complex manifolds endowed with astheno-K\"ahler metrics, satisfying certain extra differential conditions. In the second part of the paper we focus on the interplay between SKT metrics and geometrically-Bott-Chern-formal metrics. First of all, we construct a family of  simply-connected $2$-step  nilpotent Lie groups $G$, admitting discrete uniform subgroups $\Gamma$ and endowed with a left-invariant complex structure $J$, such that $(\Gamma\backslash G,J)$ carries an astheno-K\"ahler metric (see Theorem \ref{thm:5dim_astheno} for the precise statement). Such a construction will be applied in the study of the behaviour of blowups. \newline
In \cite{FT2} respectively {\cite[Proposition 2.4]{FT} it is proved that the existence of an SKT metric respectively a Hermitian metric $g$ with fundamental form $F$ on an $n$-dimensional compact complex manifold $M$, satisfying $\del\delbar F=0$, $\del\delbar F^2=0$, is stable under blowups of $M$.
In contrast, in Theorem \ref{blowup} we prove that the the existence of a Hermitian metric $g$ with fundamental form $F$  satisfying $\del\delbar F^{n-2}=0,\quad \del\delbar F^{n-3}=0$, is not preserved by  blowups.

Concerning the relation between SKT metrics and geometrically-Bott-Chern-formal metrics, we study the $6$-dimensional nilmanifolds with a left-invariant complex structure admitting a left-invariant SKT metric, which have been characterized by Fino, Parton and Salamon in \cite[Theorem 1.2]{FPS}. In particular, if we denote by {\em FPS-nilmanifold} any such a manifold, we prove the following result.
\begin{theorem*} [see Theorem \ref{thm:SKT_geom_FPS}]
Let $(M,J)$ be a $FPS$-nilmanifold. Then, any left-invariant (SKT) metric is geometrically-Bott-Chern-formal.
\end{theorem*}
Moreover, we extend this result to a class of nilmanifolds which are a generalization of FPS-manifolds in a arbitrary higher dimension (see Theorem \ref{thm:n-dim-SKT-BC-geom}).

In contrast to the mentioned positive results, on a compact complex manifold the existence of a SKT metric does not imply  the existence of geometrically-Bott-Chern-formal metrics. More precisely, we prove this for the product of a pair of certain compact complex surfaces by providing a non vanishing Aeppli-Bott-Chern-Massey product on each manifold.
\begin{theorem*}[see Theorem \ref{thm:product_BC_form}]
Let $(M,J)$ be the product of either two Kodaira surfaces, two Inoue surfaces, or a Kodaira surface and a Inoue surface. Then $(M,J)$ admits SKT metrics but does not admit geometrically-Bott-Chern-formal metrics.
\end{theorem*}
Furthermore, a similar result holds also for manifolds which are not a product of manifolds, as it is shown for a family of nilmanifolds of complex dimension $4$ in Theorem \ref{thm:4-dim_SKT_BC_form}.\vskip.2truecm

The paper is organised as follows. In Section 2, we briefly recall the notions of Dolbeault, Bott-Chern, and Aeppli cohomologies on compact complex manifolds, and we recollect the basis facts about the complex geometry of nilmanifolds. In Section 3, following \cite{Sul2} and \cite{AA}, we recall the notions of $p$-{\em pluriclosed forms} on an almost complex manifold of real dimension $2n$ (in particular, when $J$ is integrable, for $p=1$, respectively $p=n-2$, we recover the notion of SKT metrics, respectively of astheno-K\"ahler metrics). Moreover, we obtain a useful obstruction for the existence of such structures (see Lemma \ref{lemma:tecnico}), which will be used in both Proposition \ref{prop:almost-complex-application} and Theorem \ref{blowup}. Sections 4 and 5 are devoted to the proofs of Theorems  \ref{thm:5dim_astheno} and \ref{blowup}.  In Section $6$, we briefly recall the notions of Aeppli-Bott-Chern-Massey products and geometrically-Bott-Chern-formal metrics. Finally, Section 7 is mainly devoted to the proof of Theorems \ref{thm:SKT_geom_FPS}, \ref{thm:n-dim-SKT-BC-geom}, \ref{thm:product_BC_form}, \ref{thm:4-dim_SKT_BC_form}.

\vskip.2truecm
{\it Acknowledgement.} The authors kindly thank Daniele Angella and Nicoletta Tardini for useful discussions and comments. Many thanks are also due to Anna Fino for reading carefully the manuscript and for useful remarks and to Jonas Stelzig for suggesting some possible further developments.
\section{Preliminaries}\label{sec-notations}
Let $M$ be a compact $2n$-dimensional differentiable manifold endowed with an integrable almost complex structure $J$, i.e., $J\in\textrm{End}(TM)$ such that $J^2=-\id_{TM}$ and the Nijenhuis tensor associated to $J$
\begin{equation*}\label{eq:Nijenhuis}
N_J(X,Y):=[JX,JY]-[X,Y]-J[JX,Y]-J[X,JY]
\end{equation*}
vanishes for every $X,Y\in TM$. By Newlander-Nirenberg, $J$ endows $M$ with the structure of a compact complex manifold of complex dimension $n$. We will denote such manifold by $(M,J)$.

The $\C$-linear extension of the endorphism $J$ to the complexified tangent bundle $T_{\C}M:=TM\otimes \C$ gives rise to following decomposition in terms of the $\pm i$-eigenspaces of $J$
\[
T_\C M=T^{1,0}\oplus T^{0,1}M,
\]
where $T^{1,0}M:=\{X\in T_{\C}M\,:\, JX=iX \}$ and $T^{0,1}M:=\{X \in T_{\C}M\,:\, JX=-iX\}$. Such decomposition of $T_{\C} M$ extends to the exterior powers bundles $\bigwedge_\C^kM:=\bigwedge^k T_\C M$ of complex $k$-forms, that is
\[
\textstyle\bigwedge_{\C}^{k}M=\displaystyle\bigoplus_{p+q=k}\textstyle\bigwedge^{p,q}M,
\]
where $\bigwedge^{p,q}M:=\bigwedge^p T^{1,0}M\otimes \bigwedge^q T^{0,1}M$ is the bundle of {\em $(p,q)$-forms}. We will denote the global sections $\Gamma(M,\bigwedge_\C^kM)$ and $\Gamma(M,\bigwedge^{p,q}M)$  of the mentioned bundles by, respectively, $\mathcal{A}_\C^kM$ and $\mathcal{A}^{p,q}M$.

At the level of $(p,q)$-forms, the exterior differential $d$ acts as
\[
d\colon \mathcal{A}^{p,q}M\rightarrow \mathcal{A}^{p+1,q}M\oplus \mathcal{A}^{p,q+1}M,
\]
therefore, by setting $\del:=\pi^{p+1,q}\circ d$ and $\delbar:=\pi^{p,q+1}\circ d$, we obtain that $d$ splits as $d=\del+\delbar$. Since $d^2=0$, it immediately follows that $\del^2=\delbar^2=0$.

Associated to a compact complex manifold $(M,J)$, one may define the {\em de Rham, Dolbeault, Bott-Chern, and Aeppli} cohomologies,
\begin{align*}
&H_{dR}^{\bullet}(M;\C)=\frac{\Ker d}{\im d},\qquad H_{\delbar}^{\bullet,\bullet}(M)=\frac{\Ker \delbar}{\im \delbar}, \quad H_{BC}^{\bullet,\bullet}(M)=\frac{\Ker d}{\im \del\delbar},\qquad H_A^{\bullet,\bullet}(M)=\frac{\Ker \del\delbar}{\im \del +\im \delbar}.
\end{align*}
We will denote the Betti numbers by $b_k=\dim_{\C}H_{dR}^k(M;\C)$ and the Hodge numbers and the dimensions for the Bott-Chern and Aeppli cohomologies by $h_{\sharp}^{p,q}:=\dim_{\C}H_{\sharp}(M)$, for $\sharp\in\{\delbar,BC,A\}$.

From now on, $(M,J)$ will denote a compact complex manifold of complex dimension $n$.

A {\em Hermitian metric} $g$ on $(M,J)$ is a Riemannian metric on $M$ such that $J$ is an isometry with respect to $g$, i.e., $g(JX,JY)=g(X,Y)$, for every $X,Y\in \Gamma (TM)$. For any given Hermitian metric $g$, we will denote by $F$ the {\em fundamental form} of $g$, defined as $F(X,Y)= g(JX,Y)$ forn any $X,Y\in\Gamma(TM)$.

We will consider the $\C$-antilinear extension of $g$ to $\Gamma(T_\C M)$ given by\newline 
$
g(X\otimes \lambda,Y\otimes \mu):=\lambda\overline{\mu}\,g(X,Y),
$
for every $X\otimes \lambda$, $Y\otimes \mu\in \Gamma(T_\C M)$. Once fixed a Hermitian metric $g$ on $(M,J)$, then the complex cohomology groups recalled above, are isomorphic to the kernel of suitable elliptic self-adjoint operators. More precisely, setting
\begin{eqnarray*}
\Delta_{\delbar}&=& \delbar\,\delbar^* + \delbar^*\delbar\\
\Delta_{BC}&=& \del\delbar\delbar^*\del^*+
\delbar^*\del^*\del\delbar+\delbar^*\del\del^*\delbar+
\del^*\delbar\delbar^*\del+\delbar^*\delbar+\del^*\del\\
\Delta_{A}&=&\del\del^*+\delbar\delbar^*+
\delbar^*\del^*\del\delbar+\del\delbar\delbar^*\del^*+
\del\delbar^*\delbar\del^*+\delbar\del^*\del\delbar^*\,,
\end{eqnarray*}
where as usual $*:\mathcal{A}^{p,q}\to\mathcal{A}^{n-q,n-p}$ is the $\C$-linear Hodge operator defined as 
$$
\alpha\wedge *\overline{\beta} =g(\alpha,\overline{\beta})\frac{\omega^n}{n!}
$$
and $\delbar^*=-*\del *$, $\del^*=-*\delbar *$, and denoting the spaces of {\em Dolbeault harmonic}, respectively {\em Bott-Chern harmonic} and {\em Aeppli harmonic} forms by 
\begin{eqnarray*}
\mathcal{H}^{\bullet,\bullet}_{\Delta_{\delbar}}(M)&=&\{\alpha\in\mathcal{A}^{\bullet,\bullet}(M)\,\,\,\vert\,\,\, \Delta_{\delbar}\alpha=0\}\\[5pt]
\mathcal{H}^{\bullet,\bullet}_{\Delta_{BC}}(M)&=&\{\alpha\in\mathcal{A}^{\bullet,\bullet}(M)\,\,\,\vert\,\,\, \Delta_{BC}\alpha=0\}\\[5pt]
\mathcal{H}^{\bullet,\bullet}_{\Delta_{A}}(M)&=&\{\alpha\in\mathcal{A}^{\bullet,\bullet}(M)\,\,\,\vert\,\,\, \Delta_{A}\alpha=0\},
\end{eqnarray*}
we have the following complex vector spaces isomorphisms
$$
H^{\bullet,\bullet}_{\delbar}(M)\simeq \mathcal{H}^{\bullet,\bullet}_{\Delta_{\delbar}}(M), \quad
H^{\bullet,\bullet}_{BC}(M)\simeq \mathcal{H}^{\bullet,\bullet}_{\Delta_{BC}}(M), \quad
H^{\bullet,\bullet}_{A}(M)\simeq \mathcal{H}^{\bullet,\bullet}_{\Delta_{A}}(M). 
$$
It turns out that 
\begin{equation}
\begin{array}{lll}
\alpha\in\mathcal{H}^{p,q}_{\Delta_{\delbar}}(M)&\iff &
\left\{\begin{array}{ll}
\delbar\alpha &=0\\
\overline{\del}^*\alpha &=0
\end{array}
\right.
\\[30pt]
\alpha\in\mathcal{H}^{p,q}_{\Delta_{BC}}(M)&\iff& 
\left\{\begin{array}{ll}
\del\alpha &=0\\
\delbar\alpha &=0\\
\del^*\delbar^*\alpha &=0
\end{array}
\right.\label{eq:BC_harm_form}
\\[30pt]
\alpha\in\mathcal{H}^{p,q}_{\Delta_{A}}(M)&\iff &
\left\{\begin{array}{ll}
\del\delbar\alpha &=0\\
\del^*\alpha &=0\\
\delbar^*\alpha &=0
\end{array}
\right.
\end{array}
\end{equation}
We recall now some basic facts of complex geometry of {\em nilmanifolds}.
Let $M$ be a nilmanifold, that is $M=\Gamma\backslash G$, where $G$ is a simply connected nilpotent Lie group and $\Gamma$ is a lattice in $G$, with $\dim_\R M=2n$. Let $\mathfrak{g}$ be the Lie algebra of $G$. Then any given left-invariant almost complex structure $J$ on $G$ gives rise to an almost complex structure $J$ on $M$; therefore, any almost complex structure $J$ on $\mathfrak{g}$ gives rise to an almost complex structure on $M$, denoted with the same symbol $J$. According to the Newlander and Nirenberg Theorem, $J$ is a complex structure if and only 
$$
N_J(X,Y)=0\,,\quad \forall X,Y\in\mathfrak{g}.
$$
 Equivalently, an almost complex structure on $\mathfrak{g}$ can be defined by assigning an $n$-dimensional complex subspace $\mathfrak{g}^{1,0}$ of $\mathfrak{g}^*_\C$, such that $\mathfrak{g}^{1,0}\cap\overline{\mathfrak{g}^{1,0}}=\{0\}$. We set, as usual, $$\textstyle\bigwedge^{p,q}\mathfrak{g}:=\bigwedge^p\mathfrak{g}^{1,0}\otimes \bigwedge^q\overline{\mathfrak{g}^{1,0}}.$$ Then it turns out that a nilpotent Lie algebra $\mathfrak{g}$ has a complex structure if and only if there exists a basis $\{\eta^1,\ldots,\eta^n\}$ of $\mathfrak{g}^{1,0}$ such that 
$$
d\eta^{k+1}\in\mathcal{I}(\eta^{1},\ldots,\eta^k)\,
$$
where $\mathcal{I}(\eta^{1},\ldots,\eta^k)$ denotes the ideal generated by $\{\eta^{1},\ldots,\eta^i\}$ in $\bigwedge^*\mathfrak{g}^*_\C$ (see \cite[Theorem 1.3]{Sal}) and $d$ denotes the extension of the Chevalley-Eilenberg exterior differential on the Lie algebra $\mathfrak{g}$ to $\bigwedge^*\mathfrak{g}_\C^*$.
\section{$p$-pluriclosed structures}\label{p-PL}
We review the notion of positive forms on almost complex manifolds. We start by recalling some preliminary linear algebra notions.
Let $V$ be a real $2n$-dimensional vector space endowed with a complex structure $J$, that is an endomorphism $J$ of $V$ satisfying $J^2=-\hbox{\rm id}_V$. Denote by $V^*$ be the dual space of $V$ and denote by the same symbol the complex structure on $V^*$ naturally induced by $J$ on $V$. Then the complexified 
$V^{*\C}$ splits as the direct sum of the $\pm\,i$-eigenspaces, $V^{1,0}$, $V^{0,1}$ of the extension of $J$ to $V^{*\C}$, given by
$$
\begin{array}{l}
V^{1,0}=\{\eta\in V^{*\C}\,\,\,\vert\,\,\,J\eta=i\eta\}= 
\{\alpha-iJ\alpha \,\,\,\vert\,\,\,\alpha\in V^{*}\}\\[10pt]
V^{0,1}=\{\psi\in V^{*\C}\,\,\,\vert\,\,\,J\psi=-i\psi\}= 
\{\beta+iJ\beta \,\,\,\vert\,\,\,\beta\in V^{*}\},
 \end{array}
$$
that is 
$$V^{*\C}=V^{1,0}\oplus V^{0,1}.
$$
According to the above decomposition, the space $\bigwedge^r(V^{*\C})$ of complex $r$-covectors on $V^\C$ 
decomposes as 
$$
\textstyle\bigwedge^r(V^\C)=\displaystyle\bigoplus_{p+q=r}\textstyle\bigwedge^{p,q}(V^{*\C}),
$$
where 
$$
\textstyle\bigwedge^{p,q}(V^{*\C})=\bigwedge^p(V^{1,0})\otimes\bigwedge^q(V^{0,1}).
$$
If $\{\eta^1,\ldots,\eta^{n}\}$ is a basis of $V^{1,0}$, then
$$
\{\eta^{i_1}\wedge\cdots\wedge\eta^{i_p}\wedge\overline{\eta^{j_1}}\wedge\cdots\wedge\overline{\eta^{j_q}}\,\,\,\vert\,\,\, 1\leq i_1<\cdots<i_p\leq n,\,\,
1\leq j_1<\cdots<j_q\leq n\}
$$
is a basis of $\bigwedge^{p,q}(V^{*\C})$. Set $\sigma_p=i^{p^2}2^{-p}$. Then, given any 
$\eta\in\bigwedge^{p,0}(V^{*\C})$ we have that
$$
\overline{\sigma_p\eta\wedge\overline{\eta}}=\sigma_p\eta\wedge\overline{\eta},
$$
that is $\sigma_p\eta\wedge\overline{\eta}$ is a $(p,p)$-real form. Consequently, denoting by
$$
\textstyle\bigwedge_{\R}^{p,p}(V^{*\C})=\{\psi\in\bigwedge^{p,p}(V^{*\C})\,\,\,\vert\,\,\,\psi=\overline{\psi}\},
$$
we get that 
$$
\{\sigma_p\eta^{i_1}\wedge\cdots\wedge\eta^{i_p}\wedge\overline{\eta^{i_1}}\wedge\cdots\wedge\overline{\eta^{i_p}}\,\,\,\vert\,\,\, 1\leq i_1<\cdots<i_p\leq n\}
$$
is a basis of $\bigwedge_{\R}^{p,p}(V^{*\C})$. By definition, $\psi\in\bigwedge^{p,0}(V^{*\C})$ is said to be {\em simple} or {\em decomposable} if 
$$
\psi=\eta^1\wedge\cdots\wedge\eta^p,
$$
for suitable $\eta^1,\ldots,\eta^p\in V^{1,0}$.
\begin{rem}
The complex structure $J$ acts on the space of real $k$-covectors $\bigwedge^k(V^*)$ by setting, for any given $\alpha\in \bigwedge^k(V^*)$,
$$J\alpha (V_1,\ldots,V_k)=\alpha(JV_1,\ldots,JV_k).
$$
Then it is immediate to check that if $\psi\in\bigwedge_{\R}^{p,p}(V^{*\C})$ then $J\psi=\psi$. For $k=2$, the converse holds.
\end{rem}
Set
$$
\hbox{\rm Vol}=(\frac{i}{2}\eta^1\wedge\overline{\eta^1})\wedge\cdots\wedge
(\frac{i}{2}\eta^n\wedge\overline{\eta^n});
$$
then  

$$
\hbox{\rm Vol}=\sigma_n\eta^1\wedge\cdots \wedge\eta^n\wedge \overline{\eta^1}\wedge\cdots\wedge
\wedge\overline{\eta^n},
$$
that is $\hbox{\rm Vol}$ is a volume form on $V$. A real $(n,n)$-form $\psi$ is said to be {\em positive} respectively {\em strictly positive} if 
$$\psi=a{\rm Vol},
$$ 
where $a\geq 0$, respectively $a>0$.  \newline 
Let $\Omega\in\bigwedge_{\R}^{p,p}(V^{*\C})$. Then $\Omega$ is said to be {\em weakly positive} if given any non-zero simple $(n-p)$-covector $\eta$, 
the real $(n,n)$-form 
$$
\Omega\wedge\sigma_{n-p}\eta\wedge \overline{\eta}
$$
is positive. The real $(p,p)$-form $\Omega$ is said to be {\em transverse} if, given any non-zero simple $(n-p)$-covector $\eta$, the real $(n,n)$-form 
$$
\Omega\wedge\sigma_{n-p}\eta\wedge \overline{\eta}
$$
is strictly positive.



The notion of positivity on complex vector spaces can be transferred pointwise to almost complex 
manifolds. Let $(M,J)$ be an almost complex manifold of real dimension $2n$; let $\bigwedge^{p,q}(M)$ be the bundle of $(p,q)$-forms on $(M,J)$. Denote by $A^{p,q}(M):=\Gamma(M, \Lambda_J^{p, q} M)$ the space of {\em $(p,q)$-forms} on $(M, J)$ and by 
$$A_{\R}^{p,p}(M):=\{\psi\in A^{p,q}(M)\,\,\,\vert\,\,\, \psi=\overline{\psi}\}$$ 
the space of {\em real $(p,p)$-forms}.
Then the exterior differential $d$ satisfies 
$$
d(A^{p, q}(M)) \subset A_J^{p + 2, q - 1}(M) + A_J^{p + 1, q}(M) + A_J^{p, q + 1}(M) + A_J^{p-1, q + 2}(M),
$$
and, consequently, $d$ decomposes as
$$
d = \mu_J + \partial_J + \overline{\partial}_J + \overline{\mu}_J,
$$
where $\mu_J = \pi^{p + 2, q - 1} \circ d$, $\overline{\partial}_J = \pi^{p, q + 1} \circ d$. Set 
$$
d^c=J^{-1}dJ.
$$
\begin{definition}\label{p-PL-def}
 Let $(M,J)$ be an almost complex manifold of real dimension $2n$ and let $1\leq p\leq n$. A 
 $p$-{\em pluriclosed form} on $(M,J)$ is a real $dd^c$-closed transverse $(p,p)$-form $\Omega$, that is 
 $\Omega$ is $dd^c$-closed and, at every $x\in M$, $\Omega_x\in\bigwedge^{p,p}_{\R}(T_x^*M)$ is transverse. The triple $(M,J,\Omega)$ is said to be an {\em almost p-pluriclosed manifold}.
\end{definition}

Let $(M,J)$ be an $n$-dimensional complex manifold and $g$ be a Hermitian metric with fundamental form $F$. Then $d^c=i(\delbar-\del)$ and consequently 
$dd^c=2i\del\delbar$.
\begin{definition} The Hermitian metric
$g$ is said to be {\em astheno-K\"ahler} in the terminolgy by Jost and Yau \cite{JY} if 
$$\del\delbar F^{n-2}=0;$$ 
$g$ is said to be {\em strong K\"ahler with torsion}, shortly {\em SKT}, if 
$$
\del\delbar F=0.
$$ 
\end{definition}

Therefore, if $g$ is an astheno-K\"ahler metric respectively SKT metric on $(M,J)$, then $(M,J,F^{n-2})$ respctively $(M,J,F)$ is an $(n-2)$-pluriclosed respectively $1$-pluriclosed manifold. 

Following Gauduchon \cite{Ga}, a Hermitian metric
$g$ on  $(M, J)$ is said to be {\em standard} if $F^{n -1}$ is
$\partial \overline\partial$-closed. As observed in \cite{FT}, if a Hermitian metric on a $4$-dimensional compact complex manifold is at the same time SKT
and astheno-K\"ahler, then it must be also standard. 
Furthermore, in \cite[Lemma 6]{JY} a necessary condition for the
existence of astheno-K\"ahler metrics on compact complex manifolds
was provided, showing that any given holomorphic
$1$-form must be $d$-closed.

In order to recall the characterization Theorem of compact complex manifolds admitting a $p$-pluriclosed structure, we review some known facts on
positive currents. Let $M$ be an $n$-dimensional complex manifold and let
$A^{p,q}(\Omega)$ respectively ${\mathcal
D}^{p,q}(\Omega)$) be the space of $(p,q)$-forms respectively
$(p,q)$-forms with compact support on $M$. Consider the ${\mathcal C}^\infty$-topology on ${\mathcal
D}^{p,q}(M)$. The {\em space of currents} of {\em bi-dimension}
$(p,q)$ or of {\em bi-degree} $(n-p,n-q)$ is the topological dual
${\mathcal D}'_{p,q}(M)$ of ${\mathcal D}^{p,q}(M)$. A
current of bi-dimension $(p,q)$ on $M$ can be identified with
a $(n-p,n-q)$-form on $M$ with coefficients distributions. A
current $T$ of bi-dimension $(p,p)$ is said to be {\em real} if
$T(\eta) =T(\overline{\eta})$, for any $\eta\in {\mathcal
D}^{p,q}(M)$. A real current $T\in{\mathcal D}'_{p,p}(M)$ is said to be {\em strongly positive}
if,
$$
T(\Omega)
\geq 0,
$$
for every weakly positive $(p,p)$-form $\Omega$. We have the following
(see 
\cite[Theorem 2.4,(4)]{A2})
\begin{theorem}
A compact $n$-dimensional manifold $N$ has a strictly weakly positive $(p,p)$-form $\Omega$ with $\del\delbar\Omega=0$ if and only if $N$ has no strongly positive currents $T\neq 0$ of bidimension $(p,p)$, such that $T=i\del\delbar A$ for some current $A$ of bidimension $(p+1,p+1)$.  
\end{theorem}

We end this section by proving a simple yet useful lemma, which yields an obstruction to the existence of $p$-pluriclosed forms on a closed almost complex manifold.
\begin{lemma}\label{pL-sufficiente}
Let $(M,J)$ be a closed almost complex manifold of real dimension $2n$.  Let $\alpha$ be a $(2n-2p-2)$-form which is not $dd^c$-closed and such that $$(dd^c\alpha)^{n-p,n-p}=\sum c_k\psi^k\wedge\overline{\psi}^k,$$
with $\psi^k$  simple $(n-p,0)$-covectors and $c_k\neq 0$ constants having the same sign. Then $(M,J)$ does not admit a $p-$pluriclosed form.

In particular,
\begin{itemize}
\item for $p=1$, $(M,J)$ does not admit SKT metrics;
\item for $p=n-2$, $(M,J)$ does not admit astheno-K\"ahler metrics.
\end{itemize}
\begin{proof}
We prove this lemma by contradiction. Suppose there exists a $p$-pluriclosed form $\Omega$ on $(M,J)$, i.e., $\Omega$ is a $(p,p)$-real form which is $dd^c$-closed and, for every $x\in M$, $\Omega_x\in\bigwedge^{p,p}(T_xM^*)$ is trasverse. Then, let $\alpha$ be a $(2n-2p-2)$-form on $(M,J)$ as above and let us assume, for example, that each $c_k>0$. Since $M$ is closed, by Stokes theorem we have that
$$
0=\int_M d(d^c(\sigma_n\Omega\wedge\alpha))=\int_M \sigma_n \Omega \wedge dd^c\alpha=\sum _k c_k \int_M \sigma_n \Omega \wedge\psi^k\wedge\overline{\psi}^k>0,
$$
which is a contradiction. To end the proof, notice that if $F$ is an astheno K\"ahler metric on $(M,J)$, the $(n-2,n-2)$-form $F^{n-2}$ is a $(n-2)$-pluriclosed form on $(M,J)$. Analogously, if $F$ is a SKT metric on $(M,J)$, $F$ is $1$-pluriclosed form on $(M,J)$.
\end{proof}
\end{lemma}
\begin{rem}
In Lemma \ref{pL-sufficiente} the thesis on the non existence of Hermitian metrics satisfying $dd^cF=0$, for $p=1$, respectively $dd^cF^{n-2}=0$, for $p=n-2$, is still valid, without assuming the integrability of $J$.
\end{rem}
\section{Astheno-K\"ahler metrics on $5$-dimensional nilmanifolds}
We now proceed to construct a family of nilmanifolds of complex dimension $5$ endowed with a left-invariant complex structure admitting an astheno-K\"ahler metric.

Let $\{\eta^1,\ldots ,\eta^5\}$ be the set of complex forms of
type $(1,0)$, such that
\begin{equation}\label{complex-structure-equations}
\left\{ \begin{array}{lcl}
d \eta^ j &=&0, \, j = 1,\ldots,4, \\[5pt]
d \eta^ 5 &=& a_1 \, \eta^{12}+a_2 \, \eta^{13}+a_3 \, \eta^{13}+a_4 \, \eta^{1\bar{1}}+a_5 \, \eta^{1\bar{2}}+a_6 \, \eta^{1\bar{3}}+a_7 \, \eta^{1\bar{4}}\\[5pt]
&& + b_1 \, \eta^{23} + b_2 \, \eta^{24}+ b_3 \, \eta^{2\bar{1}}+ b_4 \, \eta^{2\bar{2}}+ b_5 \, \eta^{2\bar{3}}+ b_6 \, \eta^{2\bar{4}}\\[5pt]
&& + c_1 \, \eta^{34} +c_2 \, \eta^{3\bar{1}}+c_3 \, \eta^{3\bar{2}}+c_4 \, \eta^{3\bar{3}}+c_5 \, \eta^{3\bar{4}}\\[5pt]
&& + d_1 \, \eta^{4\bar{1}} +d_2 \, \eta^{4\bar{2}}+d_3 \, \eta^{4\bar{3}}+d_4 \, \eta^{4\bar{4}}
\end{array}
\right.
\end{equation}
where $a_h,b_k,c_r,c_s \in \C$, $h = 1, \ldots, 7$, $k = 1, \ldots, 6$, $r = 1, \ldots, 5$, $h = 1, \ldots, 4$ and we set as usual $\eta^{AB}=\eta^A\wedge\eta^B$. Then, setting $\mathfrak{g}^{1,0}=\Span\langle\eta^1,\ldots,\eta^5\rangle$, we obtain that 
$\mathfrak{g}_\C^*=\mathfrak{g}^{1,0}\oplus\overline{\mathfrak{g}^{1,0}}$ gives rise to an integrable almost complex structure $J$ on the real $2$-step nilpotent Lie algebra $\mathfrak{g}$. Let $G$ be the simply-connected and connected Lie group with Lie algebra $\mathfrak{g}$. 
Then, for any given choice of parameters $a_h,b_k,c_r,c_s \in \mathbb{Q}[i]$ as a consequence of Malcev's Theorem \textcolor{red}{\cite{Mal}}, there exist lattices $\Gamma\subset G$, so that $(M=\Gamma\backslash G,J)$ is a nilmanifold endowed with a complex structure $J$ with $\dim_\C M=5$. We have the following
\begin{theorem}\label{thm:5dim_astheno}
Let $M=\Gamma\backslash G$ and $J$ be the complex structure on $M$ defined by (\ref{complex-structure-equations}). Then
\begin{enumerate}
 \item[\hbox{\rm I)}] The diagonal metric $g$ on $(M,J)$ whose fundamental form is $$F=\frac{i}{2}\sum_{h=1}^5\eta^{h\bar{h}}$$  is astheno-K\"ahler  if and only if the following condition holds
 \begin{equation}\label{eq:condition_AK_general}
 \begin{array}{lll}
2\Real\,(d_4\bar{a}_4 + d_4\bar{b}_4+d_4\bar{c}_4+c_4\bar{a}_4+c_4\bar{b}_4+b_4\bar{a}_4 )&=&
\vert a_1\vert^2+\vert a_2\vert^2+\vert a_3\vert^2+\vert a_5\vert^2+\vert a_6\vert^2+\vert a_7\vert^2 +\\ [10pt]
&+&\vert b_1\vert^2+\vert b_2\vert^2+\vert b_2\vert^2+\vert b_5\vert^2+\vert b_6\vert^2+\\[10pt]
&+&\vert c_1\vert^2+\vert c_2\vert^2+\vert c_3\vert^2+\vert d_1\vert^2+\vert d_2\vert^2
\end{array}
\end{equation}
\item[\hbox{\rm II)}] Let $$a_2=a_3=a_5=a_6=a_7=b_1=b_2=b_3=b_5=b_6=c_2=c_3=c_5=d_1=d_2=d_3=0$$
Then the metric $g$ 
satisfies $dd^cF^3=0$ and $dd^cF^2=0$ if and only if the following conditions hold
\begin{equation}\label{eq:condition_AK_particular}
\left\{
\begin{array}{l}
2\,\Real\,(d_4\bar{a}_4+d_4\bar{b}_4+d_4\bar{c}_4+c_4\bar{a}_4 + c_4\bar{b}_4+b_4\bar{a}_4 )=\vert a_1\vert^2+\vert c_1\vert^2\\[10pt]
 2\,\Real\,(c_4\bar{a}_4 + c_4\bar{b}_4+b_4\bar{a}_4 )=\vert a_1\vert^2\\[10pt]
 \Real\,(c_4\bar{b}_4 - d_4\bar{a}_4)= 0\\[10pt]
  \Real\,(b_4\bar{d}_4 - c_4\bar{a}_4)= 0
 \end{array}
\right.
\end{equation}
\end{enumerate}
\end{theorem}
\begin{proof}
As for I), with the aid of Sagemath and structure equations (\ref{complex-structure-equations}), it is easy to the see that
\begin{equation}\label{ddc-5-dim}
\begin{array}{lll}
\frac{2}{3}dd^c F^{3}&=&\left(2\Real\,(d_4\bar{a}_4 + d_4\bar{b}_4+d_4\bar{c}_4+c_4\bar{a}_4+c_4\bar{b}_4+b_4\bar{a}_4 )\right.\\[10pt]
&{}&-\vert a_1\vert^2-\vert a_2\vert^2-\vert a_3\vert^2-\vert a_5\vert^2-\vert a_6\vert^2-\vert a_7\vert^2\\[10pt]
&{}&-\vert b_1\vert^2-\vert b_2\vert^2-\vert b_3\vert^2-\vert b_5\vert^2-\vert b_6\vert^2\\[10pt]
&{}&\left.-\vert c_1\vert^2-\vert c_2\vert^2-\vert c_3\vert^2-\vert d_1\vert^2-\vert d_2\vert^2 \,\,\right)\,\eta^{1234\overline{1234}},
\end{array}
\end{equation}
i.e., the metric $F$ is astheno-K\"ahler on $(M,J)$ if and only if \eqref{eq:condition_AK_general} holds.\newline
II) Under the assumption 
$$a_2=a_3=a_5=a_6=a_7=b_1=b_2=b_3=b_5=b_6=c_1=c_2=c_3=c_5=d_1=d_2=d_3=0,$$
taking into account \eqref{ddc-5-dim} and by a straightforward computation, we obtain that 
$$
dd^cF^3=0,\quad dd^c F^2=0
$$
if and only if 
$$
\left\{
\begin{array}{l}
2\Real \,(d_4\bar{a}_4 + d_4\bar{b}_4+d_4\bar{c}_4+c_4\bar{a}_4+c_4\bar{b}_4+b_4\bar{a}_4) -\vert a_1\vert^2-\vert c_1\vert^2=0\\ [10pt]
2\Real(c_4\bar{a}_4+c_4\bar{b}_4+b_4\bar{a}_4)-\vert a_1\vert^2=0\\ [10pt]
2\Real(d_4\bar{a}_4+d_4\bar{b}_4+b_4\bar{a}_4)-\vert a_1\vert^2=0\\ [10pt]
2\Real(d_4\bar{a}_4+d_4\bar{c}_4+c_4\bar{a}_4)-\vert c_1\vert^2=0\\ [10pt]
2\Real(d_4\bar{b}_4+d_4\bar{c}_4+c_4\bar{b}_4)-\vert c_1\vert^2=0.
\end{array}
\right.
$$
The last system is equivalent to \eqref{eq:condition_AK_particular}.
\end{proof}
\begin{rem} In \cite[p. 185]{STos} the authors asked for an example of a non-K\"ahler compact complex manifold which
admits both balanced and astheno-K\"ahler metrics. In \cite{FGV}, and independently in \cite{LU}, the authors constructed explicit examples of such manifolds in any dimension. As a direct application of Theorem \ref{thm:5dim_astheno}, we obtain families of $5$-dimensional complex nilmanifolds carrying both astheno-K\"ahler and balanced metrics. We apply a similar construction as in \cite[Remark 2.6]{LU}. Let
$$
\hat{F}=\frac{i}{2}(A\eta^{1\bar{1}}+\eta^{2\bar{2}}+\eta^{3\bar{3}}+\eta^{4\bar{4}}+\eta^{5\bar{5}})\,
$$
where $A$ is a positive real number. Then $d\hat{F}=0$ if and only if 
\begin{equation}\label{balaced-hat-F}
a_4+Ab_4+A c_4 +Ad_4=0,
\end{equation}
where $a_4,b_4,c_4,d_4$ are the parameters as in \eqref{complex-structure-equations}
Let $g$ be the diagonal metric whose fundamental form is 
$$
F=\frac{i}{2}(\eta^{1\bar{1}}+\eta^{2\bar{2}}+\eta^{3\bar{3}}+\eta^{4\bar{4}}+\eta^{5\bar{5}})
$$
Then, according to I) of Theorem \ref{thm:5dim_astheno}, $g$   
is astheno-K\"ahler if and only if condition \eqref{eq:condition_AK_general} holds. \newline
Take
$$
a_4=-\frac{1}{10}(1+2i),\qquad b_4=i,\qquad c_4=i,\qquad d_4=1,\qquad A=\frac{1}{10}
$$
Then, with this choice of parameters, we obtain
$$
a_4+Ab_4+A c_4 +Ad_4=-\frac{1}{10}-\frac{1}{5}i+\frac{1}{10}i+\frac{1}{10}i+\frac{1}{10}=0,
$$
that is \eqref{balaced-hat-F} is satified and so, for such a choice of parameters, $\hat{F}$ gives rise to a balaced metric on $M=\Gamma\backslash G$. A straightforward calculation yields to
$$
2\Real\,(d_4\bar{a}_4 + d_4\bar{b}_4+d_4\bar{c}_4+c_4\bar{a}_4+c_4\bar{b}_4+b_4\bar{a}_4 )=1.
$$
Therefore, the Hermitian metric $g$ is astheno-K\"ahler if and only if  
condition \eqref{eq:condition_AK_general} reads as 
\begin{equation}\label{eq:condition-F}
 \begin{array}{lll}
1&=&
\vert a_1\vert^2+\vert a_2\vert^2+\vert a_3\vert^2+\vert a_5\vert^2+\vert a_6\vert^2+\vert a_7\vert^2 +\\ [10pt]
&+&\vert b_1\vert^2+\vert b_2\vert^2+\vert b_2\vert^2+\vert b_5\vert^2+\vert b_6\vert^2+\\[10pt]
&+&\vert c_1\vert^2+\vert c_2\vert^2+\vert c_3\vert^2+\vert d_1\vert^2+\vert d_2\vert^2\,.
\end{array}
\end{equation}
One can check that there exist solutions in $\mathbb{Q}[i]$ of equation \eqref{eq:condition-F}, so that, for any given solution, the associated complex nilmanifold defined as in \eqref{complex-structure-equations} admits both a balanced metric and an astheno-K\"ahler metric.
\end{rem}
As an application of Lemma \ref{pL-sufficiente}, we provide a family of compact almost complex nilmanifolds without $2$-pluriclosed forms.
\begin{prop}\label{prop:almost-complex-application}
Let $\{\psi^1,\ldots ,\psi^4\}$ be the set of complex forms of
type $(1,0)$, such that
\begin{equation}\label{almost-complex-structure-equations}
\left\{ \begin{array}{lcl}
d \psi^ j &=&0, \quad j = 1,\ldots,3, \\[10pt]
d \psi^ 4 &=& a_1 \, \psi^{12}+a_2 \, \psi^{23}+a_3 \, \psi^{1\bar{1}}+ a_4 \, \psi^{2\bar{2}}+ a_5 \, \psi^{3\bar{3}}+a_6 \, \psi^{\bar{1}\bar{2}}+a_7 \, \psi^{\bar{2}\bar{2}},
\end{array}
\right.
\end{equation}
where $a_1,\ldots,a_7 \in \mathbb{Q}[i]$. Let $G$ be the corresponding simply-connected and connected nilpotent Lie group and $\Gamma\subset G$ be a lattice such that $N=\Gamma\backslash G$ is a compact nilmanifold. Assume that 
\begin{equation}
a_1\overline{a_2}+\overline{a_6}a_7=0 
\end{equation}
and set $a=(a_1,\ldots,a_7)$. Then $(N,J_a)$ does not admit any $2$-pluriclosed form.
\end{prop}
\begin{proof}
A straightforward calculation using \eqref{almost-complex-structure-equations} yields to 
$$
\begin{array}{lll}
\frac{i}{2}dd^c\psi^{4\bar{4}}&=&(\vert a_1\vert^2+\vert a_6\vert^2)\psi^{12\bar{1}\bar{2}}+
(\vert a_2\vert^2+\vert a_7\vert^2)\psi^{23\bar{2}\bar{3}}+
(a_1\overline{a_2}+\overline{a_6}a_7)\psi^{12\bar{2}\bar{3}}\\[10pt]
&{}& +(a_2\overline{a_1}+\overline{a_7}a_6)\psi^{23\bar{1}\bar{2}}\\[10pt]
&=&(\vert a_1\vert^2+\vert a_6\vert^2)\psi^{12\bar{1}\bar{2}}+
(\vert a_2\vert^2+\vert a_7\vert^2)\psi^{23\bar{2}\bar{3}}
\end{array}
$$
The thesis follows immediately from Lemma \ref{pL-sufficiente}.
\end{proof}
\begin{rem}
For any given $a$ such that $(a_6,a_7)\neq (0,0)$, $J_a$ is a non integrable almost complex structure on $M$. Consequently, for such an $a$, $(M,J_a)$ is an almost complex manifold with no $2$-pluriclosed forms.
\end{rem}

\section{Blow-ups of astheno-K\"ahler metrics}
By classical results and more recent ones, (see \cite{Bl,Voi,A2,FT2}), we know that, for compact complex manifolds, the property of admitting, respectively, K\"ahler, balanced, or SKT metrics, is stable blow up either in a point or along a compact complex submanifold. Regarding astheno-K\"ahler metrics, in \cite{FT}, it is proved the following result.
\begin{prop}$($\cite[Proposition 2.4]{FT}$)$
Let $(M,J,g)$ be an astheno-K\"ahler manifold of complex dimension $n$ such that its fundamental $2$-form $F$ satisfies
\begin{equation}\label{eq:cond_finotomm}
dd^cF=0, \quad dd^cF^2=0.
\end{equation}
Then both the blow-up $\tilde{M}_p$ of $M$ at a point $p\in M$ and the blow-up $\tilde{M}_Y$ of $M$ along a compact complex submanifold $Y$ admit an astheno-K\"ahler metric satisfying \eqref{eq:cond_finotomm} too.
\end{prop}

In this section, we will show that blowups of astheno-K\"ahler metrics do not preserve additional differential properties of the metric, namely we construct an example of a $5$-dimensional manifold $M$ admitting a metric $F$ satisfying
\begin{equation}\label{metric-blowup}
dd^cF^2=0,\quad dd^cF^3=0,
\end{equation}
and we will consider the blowup of such manifold along a submanifold. We will prove that such blowup does not admit any Hermitian metric $\tilde{F}$ which satisfies $dd^c\tilde{F}^2=0$ and $dd^c\tilde{F}^3=0$.

We note that if $dd^cF=0$, conditions \eqref{eq:cond_finotomm} of \cite{FT} would be verified, thus yielding stability. Therefore, when we consider a Hermitian metric $F$ which satisfies weaker conditions than \eqref{eq:cond_finotomm}, e.g., the astheno-K\"ahler condition and the differential condition $dd^cF^{n-3}=0$, in general such conditions are not stable under blowups


Now, we construct a family of $5$-dimensional compact complex nilmanifolds endowed with a Hermitian metric whose fundamental form $F$ satisfies \
	\eqref{metric-blowup} and
such that the blowup of $M$ along a suitable $3$-dimensional complex submanifold $Y$ has no Hermitian metrics satisfying \eqref{metric-blowup}. To this purpose, 
we start by considering the following nilpotent Lie group $G:=(\C^5,\ast)$, where the operation $\ast$ is defined for every $w=(w_1,w_2,w_3,w_4,w_5),z=(z_1,z_2,z_3,z_4,z_5)\in \C^5$ by
\begin{align*}
&w\ast
z:=\\&(w_1+z_1,w_2+z_2+w_3+z_3,w_4+z_4,z_5+a_1w_1z_2+a_4\overline{w}_1z_1+b_4\overline{w}_2z_2+c_1w_3z_4+c_4\overline{w}_3z_3+d_4\overline{w}_4z_4+w_5),
\end{align*}
with $a_1,a_4,b_4,c_1,c_4,d_4\in \q[i]$. We can then consider the following forms on $G$
\begin{equation*}
\begin{cases}
\eta^i=dz_i, \quad i\in\{1,2,3,4\}\\[10pt]
\eta^5=dz_5-a_1z_1dz_2-a_4\overline{z}_1dz_1-b_4\overline{z}_2dz_2-c_1z_3dz_4-c_4\overline{z}_3dz_3-d_4\overline{z}_4dz_4.
\end{cases}
\end{equation*}
It can be easily seen that $\{\eta^1,\dots,\eta^5\}$ are left invariant global forms on $G$ with structure equations
\begin{equation*}
\begin{cases}
d\eta^i=0, \quad i\in\{1,2,3,4\}\\
d\eta^5=-a_1\eta^{12}+a_4\eta^{1\overline{1}}+b_4\eta^{2\overline{2}}-c_1\eta^{34}+c_4\eta^{3\overline{3}}+d_4\eta^{4\overline{4}}.
\end{cases}
\end{equation*}
The dual left invariant complex vectors fields $\{Z_1,Z_2,Z_3,Z_4,Z_5\}$ on $G$ are given by
\begin{equation*}
\begin{cases}
Z_1=\frac{\del}{\del z_1}+a_4\overline{z}_1\frac{\del}{\del z_5}\\[10pt]
Z_2=\frac{\del}{\del z_2}(a_1z_1+b_4\overline{z}_2)\frac{\del}{\del z_5}\\[10pt]
Z_3=\frac{\del}{\del z_3}+c_4\overline{z}_3\frac{\del}{\del z_5}\\[10pt]
Z_4=\frac{\del}{\del z_4}+(c_1z_3+d_4\overline{z}_4)\frac{\del}{\del z_5}\\[10pt]
Z_5=\frac{\del}{\del z_5}.
\end{cases}
\end{equation*}
We note that $T_{\C}G=\langle Z_1,\dots,Z_5,\overline{Z}_1,\dots,\overline{Z}_5\rangle$ and the distribution $D=\langle Z_1,\dots,Z_5\rangle$ is integrable. Therefore, if we denote by $J$ the almost complex structure on $G$ for which $\{Z_1,\dots,Z_5\}$ is a frame of $(1,0)$-vector fields and $\{\eta^1,\dots,\eta^5\}$ is a coframe of $(1,0)$-forms, then $J$ is an integrable left invariant almost complex structure on $G$.

Since the constant structures $a_1,a_4,b_4,c_1,c_4,d_4$ are numbers in $\q[i]$, Malcev theorem assures the existence of a discrete uniform subgroup $\Gamma$ such that $M:=\Gamma\backslash G$ is a compact nilmanifold. In particular, since $J$ is left invariant on $G$, it descends to $M$, i.e., $(M,J)$ is a complex $5$-dimensional nilmanifold. In particular $\{Z_1,\dots,Z_5\}$ and $\{\eta^1,\dots,\eta^5\}$ are a global left invariant frame of $(1,0)$-vector fields, respectively $(1,0)$-forms on $M$.

In particular, we point out that $M$ is the nilmanifold associated to the complex Lie algebra $\mathfrak{g}$ of Section 4, with structure constants
\[
a_2=a_3=a_5=a_6=a_7=b_1=b_2=b_3=b_5=b_6=c_1=c_2=c_3=c_5=d_1=d_2=d_3=0.
\]
If we denote by $$p\colon G\rightarrow M$$ the natural quotient projection from $G$ to $\Gamma\backslash G$ and we set $$Y_0:=\{(z_1,z_2,z_3,z_4,z_5):z_2=z_4=0\}\subset G,$$ then $p(Y_0)=:Y\subset M$ is a compact complex $3$-dimensional submanifold of $M$ whose complexified tangent bundle $T_{\C}Y$ is spanned by $\{Z_1,Z_3,Z_5,\overline{Z}_1,\overline{Z}_3,\overline{Z}_5\}$.

It is immediate to check that $Y$ is a $3$-dimensional nilmanifold and $\{\eta^1,\eta^3,\eta^3\}$ is a global coframe of $(1,0)$-forms on $Y$ with complex structure equations given by
\begin{equation}\label{eq:struct_eq_Y}
\begin{cases}
d\eta^1=0,\\[5pt]
d\eta^3=0\\[5pt]
d\eta^5=a_4\eta^{1\overline{1}}+c_4\eta^{3\overline{3}}.
\end{cases}
\end{equation}
For the convenience of the reader, we set $\alpha^1:=\eta^1$, $\alpha^2:=\eta^3$, and $\alpha^3:=\eta^5$, so that we can rewrite \eqref{eq:struct_eq_Y} as
\begin{equation}
\begin{cases}
d\alpha^1=0,\\[5pt]
d\alpha^2=0\\[5pt]
d\alpha^3=a_4\alpha^{1\overline{1}}+c_4\alpha^{2\overline{2}}.
\end{cases}
\end{equation}
Now fix the following constant structures
\begin{equation*}
a_1=-1-3i,\quad a_4=1,\quad b_4=1,\quad c_1=-4,\quad c_4=2,\quad d_4=2,
\end{equation*}
and consider the metric
\[
F=\frac{i}{2}\sum_{j=1}^5\eta^j\wedge\eta^{\overline{5}}.
\] 
For such choice of coefficients, by Theorem \ref{thm:5dim_astheno}, we have that
\[
dd^cF^2=0, \quad dd^cF^3=0, \quad dd^cF\neq 0.
\]
Now, let us consider the blowup $\pi\colon \tilde{M}_Y\rightarrow M$ of $M$ along the compact complex submanifold $Y$, with $E$ the exceptional divisor. We note that $E$ has complex dimension $4$, since each fiber $\pi^{-1}(y)\subset\tilde{M}_Y$ over a point $y\in Y$ has dimension $1$ and $\dim_{\C}Y=3$.

By contradiction, now let us assume that the astheno-K\"ahler condition $dd^cF^3=0$ and the condition $dd^cF^2=0$ are stable, i.e., there exists a Hermitian metric on $\tilde{M}_Y$ such that $dd^c\tilde{F}^3=0$ and $dd^c\tilde{F}^2=0$.

Then, the restriction of $\tilde{F}$ on $E$ gives rise to a Asteno-K\"ahler metric on $E$, that is $dd^c(\tilde{F}\restrict{E})^2=0$, i.e., $E$ is a $2$-pluriclosed manifold.

We now recall the following useful proposition by Alessandrini (\cite[Proposition 3.1]{A3}), adapted here to the setting of $p$-pluriclosed manifolds.
\begin{prop}\label{prop:alessandrini}
Let $M$ and $N$ be connected compact complex manifolds, with $\dim N =
n > m = \dim M \geq 1$, and let $f\colon N\rightarrow M$ be a holomorphic submersion, where $a := n-m =
\dim f^{-1}(x), x \in M$, is the dimension of the standard fibre $F$.
If $N$ is $p$-pluriclosed for some
$p$, $a <p \leq n - 1$, then $M$ is $(p - a)$-pluriclosed.
\end{prop}

Let us consider the map $\pi\restrict{E}\colon E\rightarrow Y$. We note that $\pi\restrict{E}$ is a holomorphic submersion with $1$-dimensional fibers, therefore by Proposition \ref{prop:alessandrini}, we have that $Y$ is $1$-pluriclosed, i.e., it admits a SKT metric.

However, this is absurd by either the characterization of $3$-dimensional SKT nilmanifolds by \cite{FPS}, or Lemma \ref{pL-sufficiente}, observing that $dd^c(-\alpha^{3\overline{3}})=8\alpha^{12\overline{12}}$.
\vskip.2truecm\noindent
Summing up, we have proved the following
\begin{theorem}\label{blowup} On a compact complex manifold of dimension $n$, the existence of a Hermitian metric $F$ such that 
	$$dd^cF^{n-2}=0,\qquad dd^cF^{n-3}=0
	$$
is not preserved by blowup. 
\end{theorem}
\section{Aeppli-Bott-Chern-Massey products and geometrically-Bott-Chern-formal metrics}
In this section, we review the definitions and basic facts of the notions of Aeppli-Bott-Chern-Massey products on a compact complex manifold and of  geometrically-Bott-Chern-formal Hermitian metrics.

From rational homotopy theory, a differentiable manifold $M$ is said to be {\em Sullivan formal} if the algebra of differential forms endowd with the exterior differential $d$, i.e., the pair $(\mathcal{A}_\C^{\bullet},d)$, is equivalent to a differential graded algebra $(B,d_{B})$ with zero differential, i.e., $d_{B}\equiv 0$.

In the $50$'s, Massey introdeced certain cosets of the de Rham cohomology, the {\em Massey triple products}, which, if non trivial, yield an obstruction to Sullivan's formality.

More recently, in \cite{AT15} a notion of Massey triple products has been introducted as an adaptation of classical Massey products for the Bott-Chern cohomology of complex manifolds.

Let $(M,J)$ be a compact complex manifold, and choose
\[
a=[\alpha]\in H_{BC}^{p,q}(M),\quad b=[\beta]\in H_{BC}^{r,s}(M), \quad c=[\gamma]\in H_{BC}^{u,v}(M)
\]
such that $$a\cup b=0\in H_{BC}^{p+r,q+s}(M), \quad b\cup c=0\in H_{BC}^{q+u,s+v}(M),$$
i.e., there exists $f_{\alpha\beta}\in\mathcal{A}^{p+r-1,q+s-1}M$ and $f_{\beta\gamma}\in\mathcal{A}^{r+u-1,s+v-1}M$ such that
\[
(-1)^{p+q}\alpha\wedge\beta=\del\delbar f_{\alpha\beta}, \quad (-1)^{r+s}\beta\wedge\gamma=\del\delbar f_{\beta\gamma}.
\] 
Then, the {\em triple Aeppli-Bott-Chern-Massey product} of $a,b,c$ is given by
\[
\langle a,b,c\rangle_{ABC}:=[(-1)^{p+q}\alpha\wedge f_{\beta\gamma}-(-1)^{r+s}f_{\alpha\beta}\wedge\gamma]_{A}\in\frac{H_A^{p+r+u-1,q+s+v-1}}{H_{A}^{r+u-1,s+v-1}(M)\cup a+H_A^{p+r-1,q+s-1}(M)\cup c}.
\]
We note that such a construction does not depend on the choice of representatives $\alpha,\beta,\gamma$ or the choice of the primitives $f_{\alpha\beta}, f_{\beta\gamma}$.

Furthermore, a notion of geometric formality analogous to the one in the sense of Kotschick (see \cite{Kot}) has been defined in \cite{AT15}.

Let $g$ be a Hermitian metric on a compact complex manifold $(M,J)$ and let $\mathcal{H}_{\Delta_{BC}}^{\bullet,\bullet}(M)$ be the space of Bott-Chern harmonic forms on $(M,J)$ with respect to $g$. The Hermitian metric $g$ is said to be {\em geometrically-Bott-Chern-formal} if $\mathcal{H}_{\Delta_{BC}}^{\bullet,\bullet}(M)$ has a structure of algebra induced by the $\wedge$ product, i.e., if for every two Bott-Chern harmonic forms $\alpha\in\mathcal{H}_{\Delta_{BC}}^{p,q}(M)$, $\beta\in\mathcal{H}_{\Delta_{BC}}^{r,s}(M)$, we have that $$\alpha\wedge\beta\in\mathcal{H}_{\Delta_{BC}}^{p+r,q+s}(M),$$
in particular, by the characterization \eqref{eq:BC_harm_form}, $$\del(\alpha\wedge\beta)=0, \quad \delbar(\alpha\wedge\beta)=0, \quad \del\delbar\ast_g(\alpha\wedge\beta)=0.$$
It is clear that, since first two conditions are always satisfied by Leibnitz rule, since $\alpha$ and $\beta$ are $d$-closed, the only condition that needs to be checked (and the only one involving the metric $g$) is $\del\delbar\ast_g(\alpha\wedge\beta)=0$.

It turns out that Aeppli-Bott-Chern-Massey products are an obstruction to the existence of geometrically-Bott-Chern-formal metrics, as proved in \cite{AT15}.
\begin{theorem*}$($\cite[Theorem 2.4]{AT15}$)$.
Triple Aeppli-Bott-Chern-Massey products vanish on compact complex geometrically-Bott-Chern-formal manifolds.
\end{theorem*}
\section{Geometric Bott-Chern formality and Strong K\"ahler with Torsion metrics}
In this section we investigate the relation between the notions of SKT metrics and geometrically-Bott-Chern-formal metrics in the setting of nilmanifolds endowed with a left-invariant complex structure $J$ and a Hermitian metric $g$.

In complex dimension $3$, the existence of SKT metrics is fully characterized by Fino, Parton, and Salamon, in terms of the complex structure equation of the manifold, as we recall in the following.
\begin{theorem}\label{thm:FPS} $($\cite[Theorem 1.2]{FPS}$)$.
Let $M = \Gamma\backslash G$ be a $6$-dimensional nilmanifold with an invariant complex structure
$J$. Then the SKT condition is satisfied by either all invariant Hermitian metrics $g$ or by
none. Indeed, it is satisfied if and only if $J$ has a basis $(\alpha^i)$ of $(1, 0)$-forms such that
\begin{equation}\label{eq:struct_FPS}
\begin{cases}
d\alpha^1 = 0\\
d\alpha^2 = 0\\
d\alpha^3 = A\alpha^{\overline{1}2} + B\alpha^{\overline{2}2} + C\alpha^{1\overline{1}} + D\alpha^{1\overline{2}} + E\alpha^{12}
\end{cases}
\end{equation}
where $A,B,C,D,E$ are complex numbers such that
\begin{equation}\label{eq:FPS}
|A|^2+|D|^2+|E|^2+2\Real(\overline{B}C)=0.
\end{equation}
\end{theorem}
We will refer to $6$-dimensional nilmanifolds satisfying \eqref{eq:struct_FPS} and \eqref{eq:FPS} as \emph{Fino-Parton-Salamon-nilmanifolds}, shortly \emph{FPS-nilmanifolds}.
 
By this classification result, we are able to prove the following theorem.

\begin{theorem}\label{thm:SKT_geom_FPS}
Let $(M,J)$ be an FPS-nilmanifold. Then, any left-invariant (SKT) metric is geometrically-Bott-Chern-formal.
\end{theorem}
Before proving Theorem \ref{thm:SKT_geom_FPS}, we will need the following lemma for the $\del\delbar$ operator on this class of manifolds.
\begin{lemma}\label{lemma:tecnico}
Let $(M,J)$ be a FPS-nilmanifold. Then, \[
\del\delbar\,\restrict{\bigwedge^{p,q}\mathfrak{g}}\equiv 0.
\]
\begin{proof}[Proof. (of Lemma 6.3)]
We begin by observing that it suffices to prove that $\del\delbar\alpha^{3\overline{3}}=0.$ In fact, let us consider the left invariant $(p,q)$-form on $M$ $$\sigma:=\alpha^{i_1}\wedge\dots\wedge\alpha^{i_p}\wedge\alpha^{\overline{j}_1}\wedge\dots\alpha^{\overline{j}_q}.$$
We note that if $\sigma$ does not contain $\alpha^{3\overline{3}}$, them $\del\delbar\sigma=0$. In fact, let us consider the two cases:\\
$(1)$ $i_k\neq 3,\overline{j}_l\neq \overline{3}$ for every $k\in\{1,\dots,p\}$, $l\in\{1,\dots,q\}$. \\
$(2)$ $i_k=3$ for some $k\in\{1,\dots,p\}$ and $j_l\neq \overline{3}$ for every $l\in\{1,\dots,q\}$, or $i_k\neq 3$ for every $k\in\{1,\dots,p\}$ and $j_l=$ for some $l\in\{1,\dots,q\}$. 

In case $(1)$, by structure equations \eqref{eq:struct_FPS} we immediately have that  $\delbar\alpha^{i_{k}}=\delbar\alpha^{\overline{j}_l}=0$. Hence, by Leibnitz rule, $\del\delbar\sigma=\del(\delbar\sigma)=0$. 

For case $(2)$, let $i_k=3$ for $k\in\{1,\dots,p\}$ and $\overline{j}_l\neq\overline{3}$. Then, up to a sign change, by Leibnitz rule we have that $\delbar\sigma=\delbar\alpha^{3}\wedge\hat{\sigma}$, where $\hat{\sigma}$ is $\sigma$ from which we remove $\alpha^{3}$. Since $\delbar{\alpha^{3}}=A\eta^{\overline{1}2}+B\eta^{2\overline{2}}+C\eta^{1\overline{1}}+D\eta^{1\overline{2}}$, we can write that $$\delbar\sigma=A\eta^{\overline{1}2}\wedge\hat{\sigma}+B\eta^{2\overline{2}}\wedge\hat{\sigma}+C\eta^{1\overline{1}}\wedge\hat{\sigma}+D\eta^{1\overline{2}}\wedge\hat{\sigma}.$$ Since $\delbar\sigma$ does not contain $\alpha^3$ or $\alpha^{\overline{3}}$, once again by \eqref{eq:struct_FPS} and Leibnitz rule, we obtain $\del\delbar\sigma=\del(\delbar\sigma)=0.$ Analogous computations can be carried out when $i_k\neq 3$ for every $k\in\{1,\dots,p\}$ and $j_l=\overline{3}$ for some $l\in\{1,\dots,q\}$.

Let us then consider $\del\delbar\alpha^{3\overline{3}}$. If $g$ is any left invariant metric on $(M,J)$ with fundamental associated form
\[
F=\frac{i}{2}\sum_{k=1}^3 F_{k\overline{k}}\alpha^{k\overline{k}}+\frac{1}{2}\sum_{k<h}\left( F_{k\overline{h}}\alpha^{k\overline{h}}-\overline{F}_{k\overline{h}}\alpha_{h\overline{k}}\right)
\]
then, by the above argument $\del\delbar F=\frac{i}{2}F_{3\overline{3}}\del\delbar\alpha^{3\overline{3}}$. By Theorem, any left-invariant on $(M,J)$ is invariant therefore, $g$ is SKT, i.e., $\del\delbar F=\del\delbar \alpha^{3\overline{3}}=0$.

Therefore, up to swapping the forms and changing the sign accordingly, for a left-invariant form $\sigma=\alpha^{3\overline{3}}\wedge\hat{\sigma}$ with $\hat{\sigma}$ not containing $\alpha^3$ nor $\alpha^{\overline{3}}$, for the above arguments we have that $\del\delbar\sigma=\del\delbar(\alpha^{3\overline{3}})\wedge\hat{\sigma}=0$. Then, by linearity of the $\del\delbar$ operator, we can conclude.
\end{proof}
\end{lemma}
\begin{proof}[Proof. (of Theorem 7.2)]
First of all, we observe that the Bott-Chern cohomology of $(M,J)$ can be computed in terms of the complex of left-invariant forms. For the sake of completeness, we recall the argument. In view of Nomizu's theorem,
\[
H_{dR}^p(\mathfrak{g})\hookrightarrow H_{dR}^p(M)
\]
is an isomorphism. By  \cite[Section 4.2]{Rollenske2} (see also \cite{CF,CFP} for other results on the Dolbeault cohomology of complex nilmanifolds), we have that
\[
H_{\delbar}^{p,q}(\mathfrak{g},J)\hookrightarrow  H_{\delbar}^{p,q}(M)
\]
gives rise to an isomorphism, that is, the Dolbeault cohomology of $(M,J)$ can be computed in terms of $\bigwedge^{p,q}\mathfrak{g}$. Finally, applying \cite[Theorem 3.7]{A13}, we obtain that
\[
H_{BC}^{p,q}(\mathfrak{g},J)\hookrightarrow H_{BC}^{p,q}(M).
\]

Now, let $g$ be a left-invariant metric on $(M,J)$ with fundamental associated form $F$. We will show that $g$ is geometrically-Bott-Chern-formal.
Let us then fix two Bott-Chern harmonic forms $\beta\in\mathcal{H}_{BC}^{p,q}(M,g)$, $\gamma\in\mathcal{H}_{BC}^{r,s}(M,g)$. Then, the product $\beta\wedge\gamma$ is Bott-Chern harmonic with respect to $g$ if, and only if, $$d(\beta\wedge\gamma)=0,\quad \del\delbar\ast_{g}(\beta\wedge\gamma)=0.$$
By Leibnitz rule, $d(\beta\wedge\gamma)=0$ since both $\beta$ and $\gamma$ are Bott-Chern harmonic. Moreover, by Lemma \ref{lemma:tecnico}, $\del\delbar(\ast_g\beta\wedge\gamma)=0$, i.e., $\beta\wedge\gamma\in\mathcal{H}_{BC}^{p+r,q+s}(M,g)$. Hence, $g$ is a geometrically-Bott-Chern -formal metric on $(M,J)$.
\end{proof}
A similar result also for a class of manifolds which generalizes the FPS manifolds in higher dimensions.
\begin{theorem}\label{thm:n-dim-SKT-BC-geom}
Let $(M,J)$  be any $n$-dimensional nilmanifold on which the Bott-Chern cohomology can be computed by the complex of left-invariant complex differential forms and which admits a coframe $\{\eta^1,\dots, \eta^n\}$ of left invariant ${1,0}$-forms on $(M,J)$ with structure equations given by
\begin{equation*}
\begin{cases}
d\eta^i=0,\quad i\in\{1,\dots,n-1\},\\
d\eta^n\in \text{Span} \langle\eta^{i\overline{j}}\rangle_{i,j=1,\dots, n-1}.
\end{cases}
\end{equation*}
Then, any left-invariant SKT metric is geometrically-Bott-Chern-formal. 
\end{theorem}

\begin{proof}
In a similar fashion to proof of Theorem \ref{thm:SKT_geom_FPS}, it can be shown that, if there exists a left invariant SKT metric on $(M,J)$, then $\del\delbar\eta^{n\overline{n}}=0$, and in particular the $\del\delbar$ operator vanishes on any left-invariant form on $(M,J)$. Hence, if $g$ a SKT metric on $(M,J)$ and we take two Bott-Chern harmonic forms $\alpha\in\mathcal{H}_{BC}^{p,q}(M,g)$, $\beta\in \mathcal{H}_{BC}^{r,s}(M,g)$, then $\alpha\wedge\beta$ is closed and $\del\delbar(\ast_{g}\alpha\wedge\beta)=0$, therefore $\alpha\wedge\beta\in\mathcal{H}_{BC}^{p+r,q+s}(M,g)$, i.e., the product of two any Bott-Chern harmonic forms with respect to the SKT metric $g$ is Bott-Chern harmonic. This implies that every left-invariant SKT metric is geometrically-Bott-Chern-formal.
\end{proof}

In higher dimension and under more general conditions on the complex structure of the nilmanifold, however, similar results do not hold. Certain products of compact complex surfaces, e.g., are SKT but do not admit geometrically-Bott-Chern-formal metrics, as proved in the following theorem.

\begin{theorem}\label{thm:product_BC_form}
Let $(M,J)$ be the product of either two Kodaira surfaces, two Inoue surfaces, or a Kodaira surface and a Inoue surface. Then $(M,J)$ admits SKT metrics but does not admit geometrically-Bott-Chern-formal metrics.
\end{theorem} 
\begin{proof}
We begin by noticing that given the product of any two of the above compact complex surfaces $(M,J)=(M',J')\times(M'',J'')$, such manifold admits an SKT metric.

Let us consider the product metric $g:=g'+g''$, given by the sum of the diagonal constant metrics $g'$ and $g''$ with respect to certain coframes $\{\eta^1,\eta^2\}$ and $\{\eta^3,\eta^4\}$ on, respectively, $(M',J')$ and $(M'',J'')$ and let $$F'=\frac{i}{2}(\eta^{1\overline{1}}+\eta^{2\overline{2}}), \quad F''=\frac{i}{2}(\eta^{3\overline{3}}+\eta^{4\overline{4}})$$ be the fundamental forms associated to, respectively, $g'$ and $g''$. By a dimension argument, we have that on each factor $$\del\delbar F'=0,\quad \del\delbar F''=0.$$ Therefore, if $F:=F'+F''$, it is clear that $$\del\delbar  F=\del\delbar F'+\del\delbar F''=0,$$ i.e., the product metric $g$ is SKT on $(M,J)$.    (We will refer to such metric by $g$.)

We will show that none of the above product manifolds admits geometrically-Bott-Chern-formal metrics by exhibiting a non vanishing Aeppli-Bott-Chern-Massey product on each manifold.

Note that on each product, arguing as in the proof of \ref{thm:SKT_geom_FPS}, we are able to compute the Bott-Chern and Aeppli cohomologies via the complex of left-invariant complex forms.

$(i)$ \textbf{The product of two Kodaira surfaces of primary type.}\\
Let $(M,J)=(KT,J_{KT})\times(KT,J_{KT})$ be the product of two Kodaira surfaces. The complex structure $J$ is determined by the coframe $\{\eta^1,\eta^2,\eta^3,\eta^4\}$ of left-invariant $(1,0)$-forms such that its structure equations read
\begin{equation}\label{eq:struct_KT_KT}
\begin{cases}
d\eta^1=0\\
d\eta^2=A\eta^{1\overline{1}}\\
d\eta^{3}=0\\
d\eta^4=B\eta^{3\overline{3}},
\end{cases}
\end{equation}
for $A,B\in\C\setminus\{0\}$.

From \eqref{eq:struct_KT_KT}, it is easy to see that the following Bott-Chern cohomology classes
\[
[\eta^{1\overline{1}}]_{BC}, \quad [\eta^{3\overline{3}}]_{BC}, \quad [\eta^3]_{BC},
\]
are non zero. Also, we have that
\begin{align}
\eta^{1\overline{1}}\wedge\eta^{3\overline{3}}=\del\delbar\left(-\frac{1}{A\overline{B}}\eta^{2\overline{4}}\right), \quad \eta^{3\overline{3}}\wedge\eta^3=0.
\end{align}
Then, it is well defined the following Aeppli-Bott-Chern-Massey product
\[
\langle[\eta^{1\overline{1}}]_{BC},[\eta^{3\overline{3}}]_{BC},[\eta^{3}]_{BC}\rangle_{ABC}=\left[-\frac{1}{A\overline{B}}\eta^{23\overline{4}
}\right]_A\in\frac{H_A^{2,1}(M)}{H_A^{1,0}(M)\cup\,[\eta^{1\overline{1}}]_{BC}+H_A^{1,1}(M)\cup\,[\eta^{3}]_{BC}}.
\]
Since $d\ast\eta^{23\overline{4}}=-d(\eta^{123\overline{14}})=0$, the form $\eta^{23\overline{4}}$ is Aeppli harmonic, hence, as a cohomology class in $H_{A}^{2,1}(M)$, we have that
\[
\left[-\frac{1}{A\overline{B}}\eta^{23\overline{4}}\right]_A\neq 0.
\]
It remains to show that $\left[-\frac{1}{A\overline{B}}\eta^{23\overline{4}}\right]_A\notin H_A^{1,0}(M)\cup\,[\eta^{1\overline{1}}]_{BC}+H_A^{1,1}(M)\cup\,[\eta^{3}]_{BC}$.

Let us then suppose, by contradiction, the opposite, i.e.,
\begin{equation}\label{eq:ideal_KTxKT}
-\frac{1}{A\overline{B}}\eta^{23\overline{4}}=\sum_{i=1}^{h_A^{1,0}}r_i\xi^i\wedge\eta^{1\overline{1}}+\sum_{j=1}^{h_A^{1,1}}s_j\psi^j\wedge\eta^3+\del R+\delbar S,
\end{equation}
where $h_A^{p,q}:=\dim \mathcal{H}_{A}^{p,q}(M,g)$, $r_i,s_j\in\C$, $R\in\mathcal{A}^{1,1}(M)$, $S\in\mathcal{A}^{2,0}(M)$, and $\{\xi^i\}$ and $\{\psi^j\}$ are the left-invariant harmonic representatives of, respectively, $H_A^{1,0}(M)$,  and $H_A^{1,1}(M)$, with respect to $g$.

It is immediate to compute the invariant Aeppli cohomology of $(M,J)$ of bi-degree $(1,0)$ and $(1,1)$, resulting in
\begin{gather*}
\xi^1=\eta^1,\quad \xi^2=\eta^2,\quad \xi^3=\eta^3,\quad \xi^4=\eta^4,\\
\psi^1=\eta^{1\overline{2}},\,\,\, \psi^2=\eta^{1\overline{3}},\,\,\, \psi^3=\eta^{1\overline{4}},\,\,\,  \psi^4=\eta^{2\overline{1}},\,\,\, \psi^5=\eta^{2\overline{2}},\,\,\, \psi^6=\eta^{2\overline{3}},\,\,\, \psi^7=\eta^{3\overline{1}},\\
\psi^8=\eta^{3\overline{2}},\,\,\, \psi^9=\eta^{3\overline{3}},\,\,\, \psi^{10}=\eta^{3\overline{4}},\,\,\, \psi^{11}=\eta^{4\overline{1}},\,\,\, \psi^{12}=\eta^{4\overline{3}},\,\,\, \psi^{13}=\eta^{2\overline{4}}-\frac{A\overline{B}}{\overline{A}B}\eta^{4\overline{2}}.
\end{gather*}
Then, equation \eqref{eq:ideal_KTxKT} can be rewritten as
\begin{align}\label{eq:ideal_1_KTxKT}
-\frac{1}{A\overline{B}}\eta^{23\overline{4}}&=-r_2\eta^{12\overline{1}}-r_3\eta^{13\overline{1}}-r_4\eta^{14\overline{1}}-s_1\eta^{13\overline{2}}-s_2\eta^{13\overline{3}}-s_3\eta^{13\overline{4}}-s_4\eta^{23\overline{1}}
-s_5\eta^{23\overline{2}}\\
&-s_6\eta^{23\overline{3}}+s_{10}\eta^{34\overline{1}}+s_{11}\eta^{34\overline{3}}+s_{12}\eta^{34\overline{4}}-s_{13}\eta^{23\overline{4}}-s_{13}\frac{A\overline{B}}{\overline{A}B}\eta^{34\overline{2}}+\del R+\delbar S.\nonumber
\end{align}
We note that the form $\eta^{12\overline{134}}$ is $d$-closed. Therefore, if we multiply \eqref{eq:ideal_1_KTxKT} by $\eta^{12\overline{134}}$, we obtain
\[
0=s_{13}\frac{A\overline{B}}{\overline{A}B}\eta^{1234\overline{1234}}+\del(R\wedge\eta^{12\overline{134}})+ \delbar(S\wedge\eta^{12\overline{134}}),
\]
i.e.,
\begin{equation}\label{eq:ideal_2_KTxKT}
s_{13}\frac{A\overline{B}}{\overline{A}B}\text{Vol}=\del(-R\wedge\eta^{12\overline{134}})+\delbar(-S\wedge\eta^{12\overline{134}}).
\end{equation}
By integrating \eqref{eq:ideal_2_KTxKT} and applying Stokes theorem on a manifold with empty boundary, we obtain that $s_{13}=0.$

If we repeat the same argument, multiplying now \eqref{eq:ideal_1_KTxKT} by the $d$-closed form $\eta^{14\overline{123}}$, we obtain
\[
\frac{1}{A\overline{B}}\text{Vol}=\del(R\wedge\eta^{14\overline{123}}) +\delbar(S\wedge\eta^{14\overline{123}}),
\]
which, by integrating and Stokes theorem, leads to a contradiction.

To summarise,
\[
\langle[\eta^{1\overline{1}}]_{BC}, [\eta^{3\overline{3}}]_{BC}, [\eta^{3}]_{BC}\rangle_{ABC}\neq 0,
\]
i.e., we obtained a non vanishing Aeppli-Bott-Chern-Massey product, which, by \cite[Theorem 2.4]{AT15}, implies that $(M,J)$ does not admit geometrically-Bott-Chern-formal metrics.\\
\vspace{0.05cm}

$(ii)$ \textbf{The product of two Inoue surfaces of type $\mathcal{S}_{M}$.}\\
Let $(M,J)=(\mathcal{S}_M,J_{\mathcal{S}_M})\times (\mathcal{S}_M,J_{\mathcal{S}_M})$ be product of two Inoue surfaces of type $\mathcal{S}_M$. The complex structure $J$ is determined by the left invariant $(1,0)$-coframe $\{\eta^1,\eta^2, \eta^3,\eta^4\}$ with structure equations
\begin{equation}\label{eq:struct_SmxSm}
\begin{cases}
d\eta^1=\frac{\alpha-i\beta}{2i}\eta^{12}-\frac{\alpha-i\beta}{2i}\eta^{1\overline{2}}\\
d\eta^2=-i\alpha\eta^{2\overline{2}}\\
d\eta^3=\frac{\gamma-i\delta}{2i}\eta^{34}-\frac{\gamma-i\delta}{2i}-\frac{\gamma-i\delta}{2i}\eta^{3\overline{4}}\\
d\eta^4=-i\gamma\eta^{4\overline{4}},
\end{cases}
\end{equation}
for $\alpha,\gamma\in\R\setminus\{0\}$, $\beta,\delta\in\R$. 

From \eqref{eq:struct_SmxSm}, it is clear that the following Bott-Chern cohomology classes
\[
[\eta^{2\overline{2}}]_{BC}, \quad [\eta^{34\overline{3}}]_{BC}, \quad [\eta^{4\overline{4}}]_{BC}
\]
are well defined and non zero. Moreover,
\[
\eta^{2\overline{2}}\wedge\eta^{34\overline{3}}=\del\delbar(-\frac{1}{2\alpha\gamma})\eta^{23\overline{3}},\qquad \eta^{34\overline{3}}\wedge\eta^{4\overline{4}}=0,
\]
hence the following Aeppli-Bott-Chern-Massey product
\[
\langle[\eta^{2\overline{2}}_{BC},  [\eta^{34\overline{3}}]_{BC},  [\eta^{4\overline{4}}]_{BC}
\rangle_{ABC}=\left[\frac{1}{\alpha\gamma}\eta^{234\overline{34}}
\right]_A\in\frac{H_A^{3,2}(M)}{H_A^{2,1}\cup[\eta^{2\overline{2}}]_{BC}+H_A^{2,1}(M)\cup[\eta^{4\overline{4}}]_{BC}},
\]
is well defined.

Note that since $d(\ast_g\eta^{234\overline{34}})=-d(\eta^{12\overline{1}})=0$, the form $\eta^{234\overline{34}}$ is Aeppli-harmonic and, as a Aeppli cohomology class, $$\left[\frac{1}{\alpha\gamma}\eta^{234\overline{34}}
\right]_A\neq 0.$$
It remains to show that $\left[\frac{1}{\alpha\gamma}\eta^{234\overline{34}}
\right]_A\notin H_A^{2,1}\cup[\eta^{2\overline{2}}]_{BC}+H_A^{2,1}(M)\cup[\eta^{4\overline{4}}]_{BC}$. In order to do, we prove that $H_A^{2,1}(M)=\{0\}$, yielding that $H_A^{2,1}\cup[\eta^{2\overline{2}}]_{BC}+H_A^{2,1}(M)\cup[\eta^{4\overline{4}}]_{BC}=0$.

By definition, we observe that
\[
H_A^{2,1}(M):=\frac{\Ker(\del\delbar\restrict{\mathcal{A}^{2,1}(M)})}{\text{Im}(\del\restrict{ \mathcal{A}^{1,1}(M)})+\text{Im}(\delbar\restrict{\mathcal{A}^{2,0}(M)})}
\]
With the aid of structure equations \eqref{eq:struct_SmxSm} and Sagemath, we can compute
\begin{align*}
&\dim_{\C} \Ker(\del\delbar\restrict{\mathcal{A}^{2,1}(M)})=15\\
&\dim_{\C} \text{Im}(\del\restrict{ \mathcal{A}^{1,1}(M)})=12\\
&\dim_{\C} \text{Im}(\delbar\restrict{\mathcal{A}^{2,0}(M)})=6\\
&\dim_{\C} \text{Im}(\del\restrict{ \mathcal{A}^{1,1}(M)})\cap \text{Im}(\delbar\restrict{\mathcal{A}^{2,0}(M)})=3,
\end{align*}
so that
\begin{align*}
\dim_{\C} H_A^{2,1}(M)&=\dim_{\C} \Ker(\del\delbar\restrict{\mathcal{A}^{2,1}(M)})-\dim_{\C} (\text{Im}(\del\restrict{\mathcal{A}^{1,1}(M)})+\text{Im}(\delbar\restrict{\mathcal{A}^{2,0}(M)}))\\
&=15-(18-3)=0.
\end{align*}
Therefore, $H_{A}^{2,1}(M)=\{0\}$
and
\[
\langle[\eta^{2\overline{2}}]_{BC},  [\eta^{34\overline{3}}]_{BC},  [\eta^{4\overline{4}}]_{BC}
\rangle_{ABC}=\left[\frac{1}{\alpha\gamma}\eta^{234\overline{34}}
\right]_A\neq 0,
\]
which, by \cite[Theorem 2.4]{AT15} implies that $(M,J)$ does not admit any geometrically-Bott-Chern-formal metric.\\
\vspace{0.05cm}

$(iii)$ \textbf{The product of a Inoue surface of type $\mathcal{S}_M$ and a primary Kodaira surface.}\\
Let $(M,J)=(\mathcal{S}_M,J_{\mathcal{S}_M})\times(KT,J_{KT})$ be the product of a Inoue surface of type $\mathcal{S}_M$ and a primary Kodaira surfaces. The complex structure $J$ is determined by the coframe of keft-invariant $(1,0)$-form $\{\eta^1,\eta^2,\eta^3,\eta^4\}$ with structure equations
\begin{equation}\label{eq:struct_SMxKT}
\begin{cases}
d\eta^1=\frac{\alpha-i\beta}{2i}\eta^{12}-\frac{\alpha-i\beta}{2i}\eta^{1\overline{2}}\\
d\eta^2=-i\alpha\eta^{2\overline{2}}\\
d\eta^3=0\\
d\eta^4=\frac{i}{2}\eta^{3\overline{3}},
\end{cases}
\end{equation}
with $\alpha\in\R\setminus\{0\}$, $\beta\in\R$.

We consider the following Bott-Chern cohomology classes
\[
[\eta^{2\overline{2}}]_{BC}, \quad [\eta^{3\overline{3}}]_{BC}, \quad [\eta^3]_{BC}.
\]
They are clearly well defined and they are not zero. Moreover,
\[
\eta^{2\overline{2}}\wedge\eta^{3\overline{3}}=\del\delbar\left(\frac{2}{\alpha}\eta^{2\overline{4}}\right), \quad \eta^{3\overline{3}}\wedge\eta^{3}=0.
\]
Therefore, the Aeppli-Bott-Chern-Massey product
\[
\langle [\eta^{2\overline{2}}]_{BC}, [\eta^{3\overline{3}}]_{BC}, [\eta^3]_{BC}\rangle_{ABC}=\left[\frac{2}{\alpha}\eta^{23\overline{4}}\right]_A\in\frac{H_A^{2,1}(M)}{H_A^{1,0}(M)\cup [\eta^{2\overline{2}}]_{BC}+H_A^{1,1}(M)\cup[\eta^{3}]_{BC}}
\]
is well defined.

Note that $d\ast(\frac{2}{\alpha}\eta^{23\overline{4}})=-\frac{2}{\alpha}d(\eta^{123\overline{14}})=0$, i.e., the form $\frac{2}{\alpha}\eta^{23\overline{4}}$ is Aeppli-harmonic and
\[
\left[\frac{2}{\alpha}\eta^{23\overline{4}}\right]_A\neq 0,
\]
as a Aeppli cohomology class.

It remains to show that $\left[\frac{2}{\alpha}\eta^{23\overline{4}}\right]_A\notin H_A^{1,0}(M)\cup [\eta^{2\overline{2}}]_{BC}+H_A^{1,1}(M)\cup[\eta^{3}]_{BC}$. Let us suppose by contradiction that this is the case, i.e.,
\begin{equation}\label{eq:ideal_SmxKT}
\frac{2}{\alpha}\eta^{23\overline{4}}=\sum_{i=1}^{h_A^{1,0}}\lambda_i\xi^i\wedge\eta^{2\overline{2}}+\sum_{j=1}^{h_A^{1,1}}\mu_j\psi^j\wedge\eta^3+\del R+ \delbar S,
\end{equation}
with $\lambda_i,\mu_j\in\C$, $R\in\bigwedge^{1,1}\mathfrak{g}$, $S\in\bigwedge^{2,0}\mathfrak{g}$, and $\{\xi^i\}$ and $\{\mu^j\}$ are, respectively, a basis for $\mathcal{H}_{A}^{1,0}(M,g)$ and $\mathcal{H}_A^{1,1}(M,g)$. By structure equations \eqref{eq:struct_SMxKT}, we can compute the spaces of Aeppli-harmonic forms with respect to $g$
\[
\mathcal{H}_A^{1,0}(M,g)=\langle\eta^3\rangle,\quad \mathcal{H}_A^{1,1}(M,g)=\langle\eta^{3\overline{3}},\eta^{3\overline{4}},\eta^{4\overline{3}}\rangle.
\]
Then, equation \eqref{eq:ideal_SmxKT} becomes
\begin{equation}\label{eq:ideal_1_SmxKT}
\frac{2}{\alpha}\eta^{23\overline{4}}=-\lambda_1\eta^{23\overline{3}}+\del R+\delbar S.
\end{equation}
Since the form $\eta^{14\overline{123}}$ is $d$-closed, if we multiply \eqref{eq:ideal_1_SmxKT} by $\eta^{14\overline{123}}$, we obtain
\[
\frac{2}{\alpha}\text{Vol}=\del(-R\wedge\eta^{14\overline{123}})+\delbar(-S\wedge\eta^{14\overline{123}}),
\]
which, by integrating over $M$ and applying Stokes theorem, leads to contradiction.

Hence, we showed that
\[
\langle[\eta^{2\overline{2}}]_{BC},[\eta^{3\overline{3}}]_{BC},[\eta^3]_{BC}\rangle_{ABC}\neq 0,
\]
i.e., $(M,J)$ admits a non vanishing Aeppli-Bott-Chern-Massey product. By \cite[Theorem 2.4]{AT15}, this implies that $(M,J)$ does not admit any geometrically-Bott-Chern-formal metric.
\end{proof}
We prove one more result in this direction, showing that the existence of Aeppli-Bott-Chern-Massey products obstructs the existence of geometrically-Bott-Chern-formal metrics on a family of $4$-dimensional complex nilmanifolds which cannot be constructed as a product of two or more manifolds.

We start by considering a nilpotent the complex $4$-dimensional Lie algebra $\mathfrak{g}$ defined by the set of forms $\{\eta^1,\eta^2,\eta^3,\eta^4\}$  satisfying the following structure equations
\begin{equation*}
\begin{cases}
d\eta^1=0\\
d\eta^2=0\\
d\eta^3=A\eta^{2\overline{1}}\\
d\eta^4=B_1\eta^{1\overline{1}}+B_2\eta^{1\overline{1}}+B_3\eta^{2\overline{2}}.
\end{cases}
\end{equation*}
We will consider the natural complex structure $J$ on $\mathfrak{g}$ which arises by choosing $\{\eta^1,\eta^2,\eta^3,\eta^4\}$ as a $(1,0)$-form on $\mathfrak{g}$.

If we pick rational constants $A, B_1,B_2,B_3 \in\q$, by Malcev's Theorem we have tha the complex $4$-dimensional Lie group $G$ associated to $\mathfrak{g}$ admits a discrete uniform subgroup $\Gamma$ such that $M:=\Gamma\backslash G$ is compact and, in particular, $(M,J)$ is a complex $4$-dimensional nilmanifold.

Moreover, with the same choice of $A, B_1,B_2,B_3 \in\q$, by \cite[Corollary ]{Rollenske2} we have the following isomorphism
\[
 H_{BC}^{p,q}(\mathfrak{g},J) \rightarrow H_{BC}^{p,q}(M),
\]
i.e., the Bott-Chern cohomology of $(M,J)$ can be computed by means of the complex of lft-invariant forms on $\mathfrak{g}$.
\begin{theorem}\label{thm:4-dim_SKT_BC_form}
Let $M=\Gamma\backslash G$ be a complex $4$-dimensional nilmanifold endowed with a left-invariant complex structure $J$ determined by a coframe of $(1,0)$-forms $\{\eta^1,\eta^2,\eta^3,\eta^4\}$ with structure equations
\begin{equation}\label{eq:struct_eq_SKT_non_geom}
\begin{cases}
d\eta^1=0\\
d\eta^2=0\\
d\eta^3=A\eta^{2\overline{1}}\\
d\eta^4=B_1\eta^{12}+B_2\eta^{1\overline{1}}+B_3\eta^{2\overline{2}},
\end{cases}
\end{equation}
with $A\in\C\setminus\{0\}$, $B_i\in\C$, such that
\begin{equation}\label{eq:SKT_4dim}
|A|^2+|B_2|^2=2\Real(B_2\overline{B}_3).
\end{equation}
Then $(M,J)$ admits a SKT metric but does not admit any geometrically-Bott-Chern-formal metric. 
\end{theorem}
\begin{proof}

Let now consider the diagonal metric $g$ with fundamental associated form
\[
F=\frac{i}{2}\sum_{h=1}^4\eta^{h\overline{h}}.
\]
With the aid of \eqref{eq:struct_eq_SKT_non_geom}, we can see clearly that $g$ is SKT, i.e., $\del\delbar F=0$.

We will show that $(M,J)$ admits a non vanishing $ABC$-Massey product, which suffices to prove that there exists no geometrically-Bott-Chern-formal metric on $(M,J)$. 

Let us consider the following Bott-Chern cohomology classes
\[
[\eta^{1\overline{1}}]_{BC},\quad [\eta^{2\overline{2}}]_{BC},\quad [\eta^{2}]_{BC}.
\]
Since $$\eta^{2\overline{2}}\wedge\eta^{\overline{1}2}=0, \quad \eta^{1\overline{1}}\wedge\eta^{2\overline{2}}=\del\delbar(\frac{1}{|A|^2}\eta^{3\overline{3}}),$$ the Aeppli-Bott-Chern-Massey product
\[
\langle[\eta^{1\overline{1}}]_{BC}, [\eta^{2\overline{2}}]_{BC}, [\eta^{2}]_{BC}\rangle_{ABC}=\left[-\frac{1}{|A|^2}\eta^{23\overline{3}}\right]_A\in \frac{H_A^{2,1}(M)}{H_A^{1,0}(M)\cup[\eta^{1\overline{1}}]_{BC}+H_A^{1,1}(M)\cup[\eta^{2}]_{BC}},
\]
is well defined.

We notice that the $d\ast_g(\frac{1}{|A|^2}\eta^{23\overline{3}})=\frac{1}{|A|^2}d(\eta^{124\overline{14}})=0$, i.e., the form $\eta^{23\overline{3}}$ is Aeppli harmonic and, as a Aeppli cohomology class, we have that $[-\frac{1}{|A|^2}\eta^{23\overline{3}}]_A\neq 0$. It remains to show that $[-\frac{1}{|A|^2}\eta^{23\overline{3}}]_A\notin H_A^{1,0}(M)\cup[\eta^{1\overline{1}}]_{BC}+H_A^{1,1}(M)\cup[\eta^{2}]_{BC}$. 

Let us now suppose by contradiction that $[-\frac{1}{|A|^2}\eta^{23\overline{3}}]_A\in H_A^{1,0}(M)\cup[\eta^{1\overline{1}}]_{BC}+H_A^{1,1}(M)\cup[\eta^{2}]_{BC}$. By straightforward computations, it is easy to check that the spaces $\mathcal{H}_A^{1,0}(M,g)$ and $\mathcal{H}_{A}^{1,1}(M,g)$ are generated, respectively, by  $\langle\psi^j\rangle_{j=1}^2$  and $\langle\xi^i\rangle_{i=1}^{11}$, where
\[
\psi^1=\eta^1, \quad \psi^2=\eta^2,
\]
and
\begin{gather*}
\xi^1=\eta^{1\overline{3}},\quad \xi^2= \eta^{1\overline{4}},\quad \xi^3=\eta^{2\overline{3}},\quad \xi^4=\eta^{2\overline{4}},\\
\xi^5=\eta^{3\overline{1}},\quad \xi^6=\eta^{3\overline{2}},\quad \xi^7=\eta^{3\overline{4}},\quad \xi^8=\eta^{4\overline{1}},\\
\xi^9=\eta^{4\overline{2}}, \quad \xi^{10}=\eta^{4\overline{3}},\quad \xi^{11}=\eta^{3\overline{3}}+\eta^{4\overline{4}}.
\end{gather*}

Then, $[-\frac{1}{|A|^2}\eta^{23\overline{3}}]_A\in H_A^{1,0}(M)\cup[\eta^{1\overline{1}}]_{BC}+H_A^{1,1}(M)\cup[\eta^{2}]_{BC}$ implies that 
\begin{equation}\label{eq:eta2313_in_ideal}
-\frac{1}{|A|^2}\eta^{23\overline{3}}=\sum_{i=1}^{2}r_i\psi^i\wedge\eta^{1\overline{1}}+\sum_{j=1}^{11}s_j\xi^j\wedge\eta^{2}+\del R+\delbar S,
\end{equation}
for $r_i,s_j\in\C$, $R\in\mathcal{A}^{1,2}(M)$, $S\in\mathcal{A}^{2,1}(M)$, so that
\begin{align}\label{eq:ideal_0_4dim}
-\frac{1}{|A|^2}\eta^{23\overline{3}}=-r_2\eta^{12\overline{1}}-s_1\eta^{21\overline{3}}-s_2\eta^{12\overline{4}}
+s_5\eta^{23\overline{1}}+s_6\eta^{23\overline{2}}+s_7\eta^{23\overline{4}}\\+s_8\eta^{24\overline{1}}+s_9\eta^{24\overline{2}}+s_{10}\eta^{24\overline{3}}+s_{11}\eta^{23\overline{3}}+s_{11}\eta^{24\overline{4}}+\del R+\delbar S.\nonumber
\end{align}
We note that the form $\eta^{13\overline{123}}$ is $d$-closed, therefore, if we multiply \eqref{eq:ideal_0_4dim} by $\eta^{13\overline{123}}$, we obtain
\[
0=s_{11}\eta^{1234\overline{1234}}+\del(R\wedge\eta^{13\overline{123}})+\delbar(S\wedge\eta^{13\overline{123}}),
\]
which, by integrating and applying Stokes theorem, forces $s_{11}=0$. Equation \eqref{eq:ideal_0_4dim} reduces to
\begin{align}\label{eq:ideal1_4dim}
-\frac{1}{|A|^2}\eta^{23\overline{3}}=-r_2\eta^{12\overline{1}}-s_1\eta^{21\overline{3}}-s_2\eta^{12\overline{4}}
+s_5\eta^{23\overline{1}}+s_6\eta^{23\overline{2}}\\
+s_7\eta^{23\overline{4}}+s_8\eta^{24\overline{1}}+s_9\eta^{24\overline{2}}+s_{10}\eta^{24\overline{3}}+\del R+\delbar S\nonumber.
\end{align}
Now, the form $\eta^{14\overline{124}}$ is $d$-closed, so if we multiply \eqref{eq:ideal1_4dim} by $\eta^{14\overline{124}}$, we obtain
\[
-\frac{1}{|A|^2}\text{Vol}=\del(R\wedge\eta^{14\overline{124}})+\delbar(S\wedge\eta^{14\overline{124}}),
\]
which, by integration and applying Stokes theorem, leads to contradiction.

Therefore, we showed that
\[
\langle[\eta^{1\overline{1}}]_{BC}, [\eta^{2\overline{2}}]_{BC}, [\eta^{2}]_{BC}\rangle_{ABC}\neq 0,
\]
i.e., $(M,J)$ admits a non vanishing Aeppli-Bott-Chern-Massey product, which implies that $(M,J)$ does not admit any geometrically-Bott-Chern-formal metric.
\end{proof}

\end{document}